\documentclass{amsart}
\usepackage{amsmath, amsthm, amssymb}
\usepackage{verbatim}
\usepackage{tikz}
\usetikzlibrary{matrix}

\usepackage[mathscr]{eucal} 
\usepackage{mathrsfs} 
\usepackage{cancel}

\usepackage{enumitem, hyperref}
\makeatletter
\def\namedlabel#1#2{\begingroup
    #2%
    \def\@currentlabel{#2}%
    \phantomsection\label{#1}\endgroup
}

\renewcommand{\d}{\partial}

\newcommand{\wt}[1]{\widetilde{#1}}
\newcommand{\lc}{\left<}
\newcommand{\rc}{\right>}

\newcommand{\eps}{\epsilon}
\newcommand{\veps}{\varepsilon}

\newcommand{\al}{\alpha} 
\newcommand{\ze}{\zeta}

\newcommand{\de}{\delta}

\newcommand{\om}{\omega}

\newcommand{\vka}{\varkappa}

\newcommand{\Om}{\Omega}
\newcommand{\De}{\Delta}

\newcommand{\cP}{\mathcal{P}}

\newcommand{\cH}{\mathcal{H}}
\newcommand{\cL}{\mathcal{L}}

\newcommand{\cO}{\mathcal{O}}

\newcommand{\bB}{\mathbb{B}}
\newcommand{\bK}{\mathbb{K}}
\newcommand{\bP}{\mathbb{P}}
\newcommand{\bR}{\mathbb{R}}

\newcommand{\bC}{\mathbb{C}}

\newcommand{\bN}{\mathbb{N}}
\newcommand{\vpi}{\varpi}

\newcommand{\wh}[1]{\widehat{#1}}
\newcommand{\ov}[1]{\overline{#1}}

\newcommand{\uU}[1]{{\rm U}#1}

\newcommand{\cali}[1]{\mathscr{#1}}

\newcommand{\Cc}{\cali{C}}

\newtheorem{thm}{Theorem}
\newtheorem{prop}[thm]{Proposition}
\newtheorem{lem}[thm]{Lemma}
\newtheorem{cor}[thm]{Corollary}

\theoremstyle{definition}

\newtheorem{defn}[thm]{Definition}
\newtheorem{remark}[thm]{Remark}

\newtheorem{expl}[thm]{Example}
\newtheorem*{rmkx}{Remark}

\numberwithin{thm}{section}
\numberwithin{equation}{section}

\renewcommand{\[}{\begin{equation}}
\renewcommand{\]}{\end{equation}}

\title[H\"older Regularity of extremal functions 
]{H\"older regularity of Siciak-Zaharjuta extremal functions on compact Hermitian manifolds
}

\author{Hyunsoo Ahn} 
\address{Department of Mathematical Sciences, KAIST, 291 Daehak-ro, Yuseong-gu, Daejeon 34141, South Korea}
\email{kakapoolove@kaist.ac.kr}

\begin{document} 


\begin{abstract} 
For a compact subset in a compact Hermitian manifold, 
we prove that the H\"older continuity of the extremal function at a given point in the set is a local property 
and that the H\"older continuity of a weighted extremal function follows from the H\"older continuities of the extremal function and the weight function with a uniform density in capacity. The second result can be seen as an extension of a result of Lu, Phung and T\^o \cite{LPT21}. 
Moreover, for a compact subset in a compact Hermitian manifold, 
we prove that, both at the point level and at the global level, the H\"older continuity of the extremal function 
with the uniform density in capacity is equivalent to the local H\"older continuity property, which is also equivalent to the weak local H\"older continuity property. 
These results are generalizations of the results of Nguyen \cite{Ng24} on compact K\"ahler manifolds. We also show that the \(\mu\)-H\"older continuity property of a convex compact subset in \(\mathbb{C}^n\) and of a compact subset in \(\mathbb{C}^n\) at a star center 
imply the local \(\mu\)-H\"older continuity property of order \(\mu\) of the convex compact subset and of the compact subset at the point, respectively. 
\end{abstract}


\maketitle

\section{Introduction}
{\em Background.}
Complex Monge-Amp\`ere equations are essential for the classification of complex surfaces and in solving geometric flow problems \cite{PSS12}. Solving complex Monge-Amp\`ere equations on compact Hermitian manifolds is important in pluripotential theory because it generalizes K\"ahler geometry to broader non-K\"ahler settings and provides a framework for analyzing degenerate equations \cite{GZ17}. 
Stability and the H\"older regularity of the solutions to the equations on compact Hermitian manifolds are obtained in \cite{KN18}. 

The H\"older continuity of the solution to the complex Monge-Amp\`ere equations is the best global regularity when the density is in \(L^p\) with \(p>1\). (Let \(\|\cdot\|\) be the standard Euclidean norm on \(\mathbb{C}^n\). For each \(p\in(1,2)\) and \(\beta\in(2-2/p,1)\), 
\(u(z):=\|z\|^{\beta}\) has \(L^p\)-density, 
but \(u\) is not Lipschitz continuous at \(\mathbf{0}\).) 
Ko{\l}odziej \cite{Ko08} first proved the H\"older continuity. 
After that, 
Demailly, Dinew, Guedj, Pham, Ko{\l}odziej and Zeriahi in \cite{DDGHKZ} showed it with an explicit exponent on compact K\"ahler manifolds. After that, Ko\l odziej and Nguyen in \cite{KN19} showed the H\"older continuity on compact Hermitian manifolds with strictly positive density. Also, Ko{\l}odziej and Nguyen in \cite{KN18} showed it on compact Hermitian manifolds with non-negative density. 

In fact, Ko\l odziej and Nguyen showed more than that. They  \cite{KN18} proved that on a compact Hermitian manifold \((X,\omega)\), complex Monge-Amp\`ere equations with finite Borel measures, dominated by complex Monge-Amp\`ere measures of \(\omega\)-plurisubharmonic H\"older continuous functions (in other words, admitting global H\"older continuous
subsolutions), have H\"older continuous solutions. 
Lu, Phung and T\^o in \cite{LPT21} showed a nearly optimal H\"older exponent of the solutions to the equations in compact Hermitian manifolds, as in compact K\"ahler manifolds.

The Siciak-Zaharjuta extremal function can serve as such a subsolution when it is H\"older continuous. 
Lu, Phung and T\^o in \cite{LPT21} also showed that for a H\"older continuous weight function on a compact Hermitian manifold, its \(\omega\)-plurisubharmonic envelope becomes an \(\omega\)-plurisubharmonic H\"older continuous function. 
One of our results extends this result of \cite{LPT21} for general compact subsets of compact Hermitian manifolds. 

Recall that the Lelong class \(\mathcal{L}\) is defined by 
$$\mathcal{L}:=\{u\in PSH(\mathbb{C}^n):
\text{for some constant }c_u,\;
u(z)\leq\max\{0,\log\|z\|\}+c_u,\;z\in\mathbb{C}^n\}.$$
For a subset \(E\) in \(\mathbb{C}^n\), the classical Siciak-Zaharjuta extremal function of \(E\) introduced in \cite{Si62,Siciak81,Za76} is defined on \(\mathbb{C}^n\) by 
$$L_{E}(z) := \sup\{ u(z) : u \in \mathcal{L}, \, u|_E \leq 0\},\quad z\in \mathbb{C}^n.$$
The regularity of the extremal function of a set gives geometric information about the set itself. For example, $E$ is pluripolar if and only if $L_E^*\equiv\infty$, where \(^*\) denotes the upper semicontinuous regularization, and \(E\) is non-pluripolar if and only if \(L_E^*\in\mathcal{L}\). 
For a real-valued function $\phi$ on $E$, the weighted extremal function $L_{E,\phi}$ 
is defined by 
$$
L_{E,\phi}(z) := \sup\{u(z): u \in \mathcal{L}, u|_E \leq \phi\},\quad z\in\mathbb{C}^n.
$$

{\em Results.} Throughout this article, $(X, \omega)$ is a 
compact Hermitian manifold of complex dimension $n$. 
A function \(\phi:X\to\mathbb{R}\cup\{-\infty\}\) is called quasi-plurisubharmonic (quasi-psh for short) if \(\phi\not\equiv-\infty\) and \(\phi\) is locally written as the sum of a smooth function and a psh function. 
For the real differential operators \(d:=\partial+\bar{\partial}\) and \(d^c:=i(\bar{\partial}-\partial)\), the family of $\omega$-plurisubharmonic ($\omega$-psh for short) 
functions on \(X\) is defined by 
\[\notag
	PSH(X,\omega) := \{v :X\to \mathbb{R}\cup\{-\infty\}: v
    \text{ is quasi-psh},\,
    \omega + dd^c v \geq 0\}.
\]
If \(\phi:X\to\mathbb{R}\cup\{-\infty\}\) is quasi-psh, then there exists a constant \(c>0\) satisfying $dd^c\phi+c\omega\geq 0$, since \(X\) is compact and $PSH(X, c_1\omega)\subset PSH(X,c_2\omega)$ for \(0<c_1<c_2<\infty\).

The extremal function \(V_E=V_{\omega;E}\) of a subset $E$ in $X$ is defined by
\[\label{eq:SZ-ext}
	V_E (x)=V_{\omega;E}(x) := \sup\{ v(x): v \in PSH(X,\omega),\; v|_E \leq 0 
    \},\quad x\in X.\]
For a weight function \(\phi:E\to\mathbb{R}\), the weighted extremal function of \((E,\phi)\) is
\[\label{eq:w-SZ-ext}
	V_{E,\phi}(x)=V_{\omega;E,\phi} (x) := \sup\{ v(x): v\in PSH(X,\omega), \; v|_E \leq \phi\},\quad x\in X.\]
When the domain of a real-valued function \(\psi\) contains \(E\), we write \(V_{E,\psi|_E}\) as \(V_{E,\psi}\) for convenience.
There is a 
bijection between the Lelong class \(\mathcal{L}\) and \(PSH(\mathbb{CP}^n,\omega_{FS})\) (\cite[Chapter~8]{GZ17}). 
When \(\phi\) is a continuous function on a compact subset of \(X\), we can extend $\phi$ to a continuous function on \(X\) with the same supremum norm by Tietze's extension theorem. An extension of a H\"older continuous function from a compact subset to \(X\) is also possible with the same H\"older exponent and the same H\"older constant by Mcshane's extension \cite{Mc34}.


Our main theorems are about the H\"older continuity of extremal functions for non-pluripolar compact sets in \(X\). 
Lu, Phung and T\^o in \cite{LPT21} proved that for a given H\"older continuous real-valued function \(\phi\) on \(X\) with a H\"older exponent \(\alpha\in (0,1)\), the envelope \(V_{X,\phi}^*\) is also H\"older continuous with the same exponent \(\alpha\). 
\(V_{X,\phi}\) is continuous by \cite{Ah26} since \(V_X\equiv0\) and \(\phi\) are continuous, thus \(V_{X,\phi}^*=V_{X,\phi}\). 
However, the H\"older exponent or even H\"older continuity of the weighted extremal functions for a general non-pluripolar compact subset is not always preserved. 
We need an additional condition, called the {\em uniform density in capacity} \cite{Ng24}. 

Let \(\Omega\) be a smoothly bounded strictly pseudoconvex domain in \(\mathbb{C}^n\). 
Bedford and Taylor in \cite{BT82} first studied the relative capacity of a Borel subset \(E\) 
in \(\Omega\), defined by  
$$Cap(E,\Omega):=\sup\{\int_{E}(dd^cu)^n:u\in PSH(\Omega),\,0\leq u\leq 1\}.$$
We denote the Euclidean open ball of center \(a\in\mathbb{C}^n\) and radius \(r\in(0,\infty]\) by \(\mathbb{B}(a,r)\) and its closure by \(\bar{\mathbb{B}}(a,r)\). 
Let \(F\) be a compact subset of \(\mathbb{C}^n\). We say that \(F\) has {a uniform density in capacity} if there exist constants \(q>0\), \(\varkappa>0\), \(\delta>0\) and a Euclidean bounded open ball \(O\) in \(\mathbb{C}^n\) such that \(\bar{\mathbb{B}}(F,\delta)\) is contained in \(O\) and for each \(b\in F\),
\[\label{eq:cap-density-Cn} 
\inf_{0<r\leq \delta} \frac{Cap(F \cap \bar{\mathbb{B}}(b,r),O)}{r^q} \geq \varkappa.\]
We say that \(F\) has {a uniform density in capacity at \(b\in F\)} if \eqref{eq:cap-density-Cn} holds.
Recall that the global Bedford-Taylor capacity of a Borel subset \(E\) 
in \(X\) is defined as
\[\label{eq:BT-cap}Cap_{\omega}(E):=\sup\{\int_E(\omega+dd^cv)^n:v\in PSH(X,\omega),\,0\leq v\leq 1\}.\]

By \cite[Proposition~2.3]{DK12}, \(Cap_{\omega}(X)<\infty\). The property of the uniform density in capacity of a compact subset in \(X\) is defined in Section~\ref{sec:hcpandudc} using the global Bedford-Taylor capacity, similarly as that of a compact subset in \(\mathbb{C}^n\).

Our main results are about the H\"older continuity of extremal functions on \(X\). 
\begin{thm}\label{thm:characterization-hcp} Let $K$ be a compact  subset of a compact Hermitian manifold \(X\) and $a\in K$. Let $\bar{B}(a,r)$ be a closed holomorphic coordinate ball in \(X\) with the center \(a\in K\) and the finite radius \(r>0\).
Then the following items hold:
\begin{itemize}
\item[(i)] Let \(\mu\in(0,1]\).
$V_K$ is $\mu$-H\"older continuous at $a$ if and only if $V_{K \cap \bar{B}(a,r)}$ is $\mu$-H\"older continuous at $a$.
\item[(ii)] If $K$ has a uniform density in capacity and   $\phi:X\to\mathbb{R}$ and \(V_K\) are H\"older continuous, then $V_{K,\phi}$ is  H\"older continuous.
\end{itemize}
\end{thm}

\begin{thm}\label{thm:characterization-hcp-hcp} Let $K$ be a compact  subset of a compact Hermitian manifold \(X\) and $a\in K$. 
Then the following three items are equivalent. 
\begin{itemize}
\item[(i)] \(V_K\) is H\"older continuous at \(a\) and \(K\) has a uniform density in capacity at \(a\). 
\item[(ii)] \(K\) has local H\"older continuity property at \(a\). 
\item[(iii)] \(K\) has weak local H\"older continuity property at \(a\). 
\end{itemize}
Also, the following three items are equivalent.
\begin{itemize}
\item[(iv)] \(V_K\) is H\"older continuous and \(K\) has a uniform density in capacity.
\item[(v)] \(K\) has local H\"older continuity property.
\item[(vi)] \(K\) has weak local H\"older continuity property.
\end{itemize}
\end{thm}

In Theorem~\ref{thm:characterization-hcp-hcp}, 
the equivalence of (iv), (v) and (vi) does not directly comes from that of (i), (ii) and (iii). 
The proof of Theorem~\ref{thm:characterization-hcp}-(i) gives Corollary~\ref{cor:sharp-Holder-norm}, which is used for the proofs of Theorem~\ref{thm:characterization-hcp}-(ii) and Theorem~\ref{thm:characterization-hcp-hcp}. 
The local H\"older continuity property and the weak local H\"older continuity property in Theorem~\ref{thm:characterization-hcp-hcp} are introduced in Section~\ref{sec:hcpandudc}. 
Theorem~\ref{thm:characterization-hcp}-(ii) is a generalization of the above-mentioned result of \cite{LPT21} to a general weighted extremal function,
with the additional condition uniform density in capacity. However, unlike the envelopes, the H\"older exponent becomes smaller: if \(V_K\) and \(\phi\) are \(\mu\)-H\"older continuous for \(\mu\in(0,1]\) and \(K\) has a uniform density in capacity of order \(q\in(0,\infty)\), then \(V_{K,\phi}\) is \(\mu^2/(\mu+q+2)\)-H\"older continuous.
N.~C. Nguyen in \cite{Ng24} proved Theorem~\ref{thm:characterization-hcp} and some analogue of Theorem~\ref{thm:characterization-hcp-hcp} for compact K\"ahler manifolds, and we generalize his results to compact Hermitian manifolds.

Let \(0<\mu,\nu\leq1\) and \(0<q,q_0\) be constants. 
In Theorem~\ref{thm:characterization-hcp-hcp}, 
the weak local \(\mu\)-H\"older continuity property of order \(q\) 
implies the uniform density in capacity of order \(nq\) 
and the \(2\mu/(q+2)\)-H\"older continuity. 
Conversely, the \(\nu\)-H\"older continuity 
and the uniform density in capacity of order \(q_0\) 
imply the local \(\nu\)-H\"older continuity property of order \(2q_0+2\). 
Also, the \(\mu\)-H\"older continuity property of order \(q\) implies the weak local \(\mu\)-H\"older continuity property of order \(q\) by the definitions. Conversely, by applying Theorem~\ref{thm:characterization-hcp-hcp} twice, we know that the weak local \(\mu\)-H\"older continuity property of order \(q\) implies the local \(2\mu/(q+2)\)-H\"older continuity property of order \(2nq+2\). In Remark~\ref{rmk:local-mu-hcp}, we directly prove that the weak local \(\mu\)-H\"older continuity property of order \(q\) directly implies the local \(\mu\)-H\"older continuity property of order \((n+1)q\). 

\begin{remark}\label{rmk:example of unifdenincap}
There are several compact subsets in \(\mathbb{C}^n\)
with uniform densities in capacity. Since \(\lim_{r\to0^+}r^{\epsilon}\log{r}=0\) for each \(\epsilon\in(0,\infty)\), closed Euclidean balls with finite radii have uniform densities in capacities by \cite[Example~5.1.1]{Kl91} and \cite[Example~4.37]{GZ17}. 
The image of a compact subset in \(\mathbb{C}^n\) with a uniform density in capacity by an automorphism of \(\mathbb{C}^n\) also has a uniform density in capacity. 
Compact subsets of \(\mathbb{C}^n\) with local H\"older continuity property have uniform densities in capacity by Lemma~\ref{lem:cap-density}, which is \cite[Lemma~3.3]{Ng24} proved by Nguyen. In \cite[Example~4.5]{Ng24}, he gives several compact subsets in \(\mathbb{C}^n\) with local H\"older continuity property, thus they have uniform densities in capacity. 

For \(\mathbb{K}=\mathbb{R}\) or \(\mathbb{K}=\mathbb{C}\), 
the closure of a bounded Lipschitz domain in \(\mathbb{K}^n\) and the closure of a bounded domain in \(\mathbb{K}^n\) with the geometrical condition, explained in Definition~\ref{defn:geometrical condtion}, are shown to have the local H\"older continuity property as subsets of \(\mathbb{C}^n\) in \cite[Example~4.5]{Ng24}. In particular, the real cube \([0,1]^n\), the polydisk \(\bar{\mathbb{B}}({0},1)^n\) and a convex compact subset of \(\mathbb{K}^n\) with the nonempty interior in \(\mathbb{K}^n\) has a uniform density in capacity in \(\mathbb{C}^n\). 
In \cite[Example~4.5]{Ng24}, the uniform interior sphere condition implies the geometrical condition since the union of uniform closed balls conjecture was proved by Nour, Stern and Takche \cite{NST11}. 
\cite[Lemma~4.7-(a)]{Ng24} directly showed that the closure of a bounded domain in \(\mathbb{C}^n\) with a uniform interior sphere condition implies the local H\"older continuity property with improvement in both the exponent and the order compared to the  geometric condition. This direct proof is also valid since the uniform interior \(r\)-sphere condition implies the union of uniform closed \(r/2\)-balls by \cite{NST11}. 
\end{remark}

\begin{remark}\label{rmk:difficulty compared to Kahler}
The main difficulty of Theorem~\ref{thm:characterization-hcp-hcp} compared to the case for compact K\"ahler manifolds is the  inequality \eqref{eq:cap-comparison}, which compares the supremum of the extremal function and the global Bedford-Taylor capacity on \(X\). 
Whether a Hermitian metric \(\omega\) on \(X\) is uniformly non-collapsing remains an open problem. Although Guedj and Lu \cite[Proposition~3.4]{GL22} proved that 
$$0<\inf\{\int_X(\omega+dd^cv)^n:v\in PSH(X,\omega),\;0\leq v\leq M\}$$
for each \(M\in(0,\infty)\), but an explicit uniform lower bound for the infimum in terms of \(X,\omega\) and \(M\) is not known. To avoid this difficulty, we compare the supremum of the extremal function and the relative capacity in \(\mathbb{C}^n\) to prove Theorem~\ref{thm:characterization-hcp-hcp}.
\end{remark}

Theorem~\ref{thm:characterization-hcp} gives a criterion to determine the following {\em\((
{C}^{\alpha},{C}^{\alpha'})\)-regularity} given by T.~C. Dinh, X.~N. Ma, and V.~A. Nguyen in \cite{DMN}.

\begin{defn} \label{defn:DMN-regular}
For \(0<\alpha,\alpha'\leq 1\), a compact subset  $K$ in a compact Hermitian manifold \(X\) is called {$({C}^\alpha,{C}^{\alpha'})$-regular} if the weighted extremal function $V_{K,\phi}$ of \((K,\phi)\) is $\alpha'$-H\"older continuous for any \(\alpha\)-H\"older continuous weight \(\phi\) and each bounded subset in the H\"older space \({C}^{\alpha}(K)\) is sent by the map \(\phi\mapsto V_{K,\phi}\) to a bounded subset in \({C}^{\alpha'}(X)\). 
\end{defn}
\begin{thm}\label{thm:UPC-intro}  
Let \(K\) be a compact subset of a compact Hermitian manifold \(X\). If \(K\) has weak local H\"older continuity property, then \(K\) is \(({C}^{\alpha(K)},{C}^{\alpha'(K)})\)-regular.
\end{thm}
The proof of Theorem~\ref{thm:UPC-intro} uses Theorem~\ref{thm:characterization-hcp} and Theorem~\ref{thm:characterization-hcp-hcp}. It also shows that, if \(K\) has weak \(\mu\)-H\"older continuity of order \(q\) for constants \(\mu\in(0,1]\) and \(q>0\), then $$\alpha(K):=\frac{2\mu}{q+2}\quad\text{and}\quad \alpha'(K):=\frac{4\mu^2}{(q+2)(2\mu+(q+2)(nq+2))}$$ are valid.

The H\"older continuity property of a general compact subset in \(\mathbb{C}^n\), introduced in Section~\ref{sec:hcpandudc}, does not imply the local H\"older continuity property, introduced in Section~\ref{sec:hcpandudc}, of the subset. However, if a compact subset in \(\mathbb{C}^n\) is convex, then we prove that the H\"older continuity property of the subset implies the local H\"older continuity property of it with the same exponent and the order equal to the exponent. We also prove that the H\"older continuity property of a compact subset at a star center of the set implies the local H\"older continuity property of the set at the point with the same exponent and the order equal to the exponent. 
\begin{thm}\label{thm:HCP implies localHCP in CN}
Let \(F\) be a compact subset of \(\mathbb{C}^n\) and let \(\mu,\nu\in (0,1]\) be constants. If \(F\) is convex and has \(\mu\)-H\"older continuity property, then \(F\) has local \(\mu\)-H\"older continuity property of order \(\mu\). If \(F\) has \(\nu\)-H\"older continuity property at \(b\in F\) and \(b\) is a star center of \(F\), i.e., \(tF+(1-t)b\subset F\) for each \(t\in[0,1]\), then \(F\) has local \(\nu\)-H\"older continuity property of order \(\nu\) at \(b\). 
\end{thm}
Nguyen \cite[Corollary~4.6]{Ng24} proved that a compact subset \(F\) in \(\mathbb{K}^n\), where \(\mathbb{K}\) is either \(\mathbb{R}\) or \(\mathbb{C}\), with the following geometrical condition in \(\mathbb{K}^n\) has local \(1/2\)-H\"older continuity property of order \(n\). 
\begin{defn}\label{defn:geometrical condtion}
Let \(F\) be a compact subset of \(\mathbb{K}^n\), where \(\mathbb{K}\) is either \(\mathbb{R}\) or \(\mathbb{C}\). \(F\) is said to satisfy the geometrical condition if there exists a constant \(r_F>0\) such that, for each \(a\in F\), there exists \(a'\in F\) allowing the convex hull of \(\{a\}\cup\bar{\mathbb{B}}_{\mathbb{K}^n}(a',r_F)\) to be contained in \(F\). Here, \({\bar{\mathbb{B}}}_{\mathbb{K}^n}(a',r_F)\) is the closed Euclidean ball in \(\mathbb{K}^n\) centered at \(a'\) of radius \(r_F\).
\end{defn}
Then by Theorem~\ref{thm:HCP implies localHCP in CN}, we get the following corollary. 
\begin{cor}\label{cor:convx with int has localHCP}
Let \(\mathbb{K}\) be either \(\mathbb{R}\) or \(\mathbb{C}\). Then 
a convex compact subset of \(\mathbb{K}^n\) with nonempty interior in \(\mathbb{K}^n\) has local \(1/2\)-H\"older continuity property of order \(1/2\).
\end{cor}

\bigskip
{\em Organization.} In Section~\ref{sec:c}, we give several basic properties of extremal functions on a compact Hermitian manifold \(X\). 
In Section~\ref{sec:hcpandudc}, we introduce the local H\"older continuity property, the weak local H\"older continuity property 
and the uniform density in capacity of a compact subset in \(X\). 
After the introduction of various H\"older continuity properties, we prove Theorem~\ref{thm:HCP implies localHCP in CN}. 
We also study the relation of extremal functions of compact subsets in \(\mathbb{C}^n\) with uniform densities in capacity of the sets. 
In Section~\ref{sec:hcp}, with the help of Demailly's estimates, we verify that the global \(\mu\)-H\"older continuity of the extremal function of a compact subset in \(X\) is implied by the \(\mu\)-H\"older continuity on the set itself. Finally, in Section~\ref{sec:pfs of thms}, we prove Theorem~\ref{thm:characterization-hcp}, Theorem~\ref{thm:characterization-hcp-hcp} and 
Theorem~\ref{thm:UPC-intro}. 
Section~\ref{sec:appendix} is an appendix including the lemma explaining the relation between extremal functions on \(X\) and locally defined weighted relative extremal functions in \(\mathbb{C}^n\), which is used in the proof of Theorem~\ref{thm:characterization-hcp-hcp}. 

\bigskip
{\em Acknowledgement.} I would like to thank my advisor Ngoc Cuong Nguyen with my deep gratitude for helping me write this article. This work was supported by the National Research Foundation of Korea (NRF) funded by the Korean government (MSIT) RS-2026-25470686.

\section{Basic properties of extremal functions}\label{sec:c}
Throughout this article, \((X,\omega)\) is a compact Hermitian manifold of complex dimension \(n\). We impose no additional assumption on the Hermitian metric \(\omega\).

The basic properties introduced in this section are the analogues for the extremal functions of the compact sets in \(\mathbb{C}^n\). For the details of these basic properties, we refer to \cite{Ah26}. 
Note that for subsets \(E\) and \(O\) of \(X\) (resp. of \(\mathbb{C}^n\)), if \(O\) is open in \(X\) (resp. in \(\mathbb{C}^n\)), then \(\sup_O V_E=\sup_O V_E^*\) (resp. \(\sup_O L_E=\sup_O L_E^*\)). 

The following proposition by Vu \cite[Lemma~2.2, 2.7]{Vu19}
gives \(\omega\)-psh functions. For an open subset \(U\) of \(X\), $PSH(U,\omega)$ is defined by $$PSH(U,\omega):=\{v:U\to\mathbb{R}\cup\{-\infty\}:v\text{ is quasi-psh},\,\omega+dd^cv\geq0\}.$$
\begin{prop}\label{prop:envelope}
(i) If \(U_1, U_2\subset X\) are open with \(\overline{U_1}\subset U_2\), \(v_1\in PSH(U_1,\omega)\) and \(v_2\in PSH(U_2,\omega)\) with \(\limsup_{U_1\ni y\to x}v_1(y)\leq v_2(x)\) for each \(x\in\partial U_1\), then \(v:U_2\to[-\infty,\infty)\)
defined as \(\max\{v_1,v_2\}\) on \(U_1\) and \(v_2\) on \(U_2\setminus U_1\) is in \(PSH(U_2,\omega)\).\\
(ii) If \(\{v_{\beta}\}_{\beta\in I}\subset PSH(X,\omega)\) are  uniformly bounded above, then \((\sup_{\beta\in I} v_{\beta})^*\in PSH(X,\omega)\).
\end{prop}

Pluripolar sets and \(\omega\)-pluripolar sets are also defined on \(X\). 
\begin{defn}
Let \(E\) be a subset of \(X\). \(E\) is called pluripolar if for each \(p\in E\), there exists open \(U\ni p\) and \(\phi\in PSH(U)\) such that \(\phi\not\equiv-\infty\) and \(E\cap U\subset \{\phi=-\infty\}\). \(E\) is called \(\omega\)-pluripolar if \(E\subset\{v=-\infty\}\) for some \(v\in PSH(X,\omega)\).
\end{defn}

We have the following equivalence due to Vu in \cite[Theorem~1.1]{Vu19}.
\begin{prop}\label{lppgpp}
\(E\subset X\) is pluripolar if and only if it is \(\omega\)-pluripolar.
\end{prop}

Characterization of pluripolar sets using their extremal functions is possible. 
\begin{prop}\label{prop:pluripolarextremal}
Let \(E\) be a subset of \(X\). 
\begin{itemize}
\item[(i)] E is pluripolar \(\Leftrightarrow\) \(\sup_X V_E^*=\infty\) \(\Leftrightarrow\) \(V_E^*\equiv\infty\).
\item[(ii)] If E is not pluripolar, then \(V_E^*\in PSH(X,\omega)\) and \(V_E^*|_{int(E)}\equiv0\).
\end{itemize}
\end{prop}

We also have the following properties as in 
\cite[Proposition~9.19]{GZ17}.
\begin{prop}\label{prop:elementary} 
Let \(E,F,P\) be subsets of \(X\). 
\begin{enumerate}
\item[(i)]
$E\subset F$ implies $V_F \leq V_E$.
\item[(ii)]
If $E$ is open, $V_E = V^*_{E}$.
\item[(iii)]
If $P$ is pluripolar, $V_{E\cup P}^* = V_E^*$.
\item[(iv)]
If $E_j\subset X$ is an increasing sequence and $E$ is the limit, 
$\lim_{j\to \infty}V_{E_j}^* = V_E^*.$
\item[(v)]
If compact $K_j\subset X$ is a decreasing sequence and $K$ is the limit,\\
$\lim_{j\to \infty} V_{K_j} = V_K$ and $\lim_{j\to \infty} V_{K_j}^*  = V_K^*$ a.e..
\end{enumerate}
\end{prop}

The lower semicontinuity of the extremal functions on \(X\) for compact sets holds.  
\begin{lem} \label{lem:lower-semicontinuity} If \(K\) is a compact subset of \(X\), then \(V_K\) is lower semicontinuous. 
\end{lem}

\begin{cor}\label{cor:semicontinuity} If \(K\) is a compact subset of \(X\), then $V_K$ is continuous if and only if $V_K^*=0$ on \(K\).
\end{cor}

\begin{remark}\label{rmk:regularity-c} Corollary~~\ref{cor:semicontinuity} implies that the 
continuity of $V_K$ for compact \(K\) does not depend on the Hermitian metric $\omega$ since for another hermitian metric $\omega'$, $\omega' \leq c \omega$ for some constant \(c>0\), $V_{\omega';K}^* \leq c V_{\omega;K}^*=cV_{K}$. Later, Lemma~\ref{lem:property} shows that the H\"older continuity of \(V_K\) and the H\"older exponent are also independent of \(\omega\). 
\end{remark}

The (zero-one) relative extremal function for a set $E$ in $X$ is defined by
\[\label{eq:r-ext}
	h_E (z) := \sup\left\{ v(z): v\in PSH(X,\omega),\; v\leq 1, 
    \;v|_E\leq 0\right\}.
\]
By Proposition~\ref{prop:envelope}, 
$h_E^* \in PSH(X,\omega)$. 
Lemma~\ref{lem:BTcap_globalandlocal} proves that \(E\) is pluripolar if and only if \(h_E^*\equiv 1\) on the compact Hermitian manifold \(X\).
For a non-pluripolar compact set $K\subset X$, denote $M_K := \sup_{X}V_K^*=\sup_{X}V_K<\infty$. Then we have 
\[\label{eq:zero-one-function}
	V_K^* \leq M_K h_K^*.\]


The global Bedford-Taylor capacity is comparable (bi-Lipschitz equivalent) to the local Bedford-Taylor capacity. 
The proof is similar to the compact K\"ahler case as in \cite[page 52-53]{Ko05}, \cite[Proposition~9.8]{GZ17}. This comparability, Proposition~\ref{lppgpp}, \cite[Corollary~4.36]{GZ17} and \cite[Theorem~4.40]{GZ17} together characterize Borel pluripolar sets and pluripolar sets via the global Bedford-Taylor capacity and the outer global Bedford-Taylor capacity respectively. 
Characterization of a pluripolar set using its relative extremal function is also possible. 
The outer global Bedford-Taylor capacity of a subset \(E\) in \(X\) is defined by  
$$Cap_{\omega}^*(E):=\inf\{Cap_{\omega}(O):O\;\text{open, }E\subset O\subset X\}.$$

\begin{lem}\label{lem:BTcap_globalandlocal} Let \(\rho_j\), \(1\leq j\leq N<\infty\), be finitely many smooth strictly psh functions 
on some holomorphic charts of \(X\) properly containing the closures of \(U_j=\{\rho_j<0\}\neq\emptyset\) respectively. Assume for some constant \(\delta>0\), \(U_{j,\delta}=\{\rho_j<-\delta\}\) cover \(X\).
Then there exist some constants \(C_j=C_j(\rho_j,X,\omega)>1\),  \(C_j'=C_j'(\rho_j,X,\omega,\delta)>1\) such that 
for each \(1\leq j_0\leq N\) and Borel subset \(E\subset X\),
\[\begin{aligned} 
\label{eq:omcap-under-relcap}
Cap_{\omega}(E\cap U_{j_0,\delta})&\leq C_{j_0} Cap(E\cap U_{j_0,\delta},U_{j_0}),\\
Cap_{\omega}(E)
\leq\sum_{j=1}^NCap_{\omega}({E\cap U_{j,\delta}})
&\leq(\max_{1\leq j\leq N}C_j)\sum_{j=1}^NCap({E\cap U_{j,\delta}},U_j),\end{aligned}\]
\[\label{eq:relcap-under-omcap}
\begin{aligned}
Cap(E\cap U_{j_0,\delta},U_{j_0})
&\leq C_{j_0}' Cap_{\omega}(E\cap U_{j_0,\delta})
\leq C_{j_0}' Cap_{\omega}(E),\\
\sum_{j=1}^N Cap(E\cap U_{j,\delta},U_j)
&\leq (\sum_{j=1}^N C_j') Cap_{\omega}(E).
\end{aligned}\]
Consequently, a Borel subset \(P_1\subset X\) is pluripolar if and only if \(Cap_{\omega}(P_1)=0\), and a subset \(P_2\subset X\) is pluripolar if and only if \(Cap_{\omega}^*(P_2)=0\) if and only if \(h_{P_2}^*\equiv 1\).
\end{lem}

For a non-pluripolar compact subset \(K\subset X\), Vu \cite[Proposition~2.9]{Vu19} obtained the following estimate of \(M_K\) in \eqref{eq:zero-one-function}
using the global Bedford-Taylor capacity. 
\begin{lem}\label{lem:BTcapATcap}
There exists a uniform positive constant \(A=A(X,\omega)\) depending only on \((X,\omega)\) such that, for every non-pluripolar compact subset \(K\) in \(X\), 
\[\label{eq:cap-comparison}0<\frac{(\int_X(dd^cV_K^*+\omega)^n)^{\frac{1}{n}}}{[Cap_{\omega}(K)]^{\frac{1}{n}}}\leq\max\{1,\sup_X V_K\},\quad \sup_X V_K\leq\frac{A}{Cap_{\omega}(K)}.\]
\end{lem}

For subsets \(E_1\subset E_2\) in \(X\) and real-valued functions \(\phi_1\leq \phi_2\) on \(X\), there is also monotonicity for weighted extremal functions as
\[ \label{eq:monotonicity}
	V_{E_2, \phi_2} \leq V_{E_1,\phi_1} \leq V_{E_1, \phi_2}.
\]
The following lemma tells other basic properties of weighted extremal functions. 
\begin{lem}\label{lem:weight-vs-unweight} 
For a compact subset \(K\) in \(X\) and a bounded function \(\phi:X\to\mathbb{R}\),
\begin{itemize}
\item[(i)] $V_K^* + \inf_K \phi \leq V_{K,\phi}^* \leq V_K^* + \sup_K \phi.$
\item[(ii)] for the current $\theta := \omega + dd^c \phi$
of bidegree \((1,1)\) and
the function $V_{\theta;K} := \sup\{v \in PSH(X, \theta): v|_K \leq 0\}$ where \(PSH(X,\theta):=\{v:X\to\mathbb{R}\cup\{-\infty\}:
v\text{ is quasi-psh},\,\theta+dd^cv\geq0\}\), $V_{K, \phi}=V_{\omega;K,\phi} = V_{\theta;K} + \phi.$
\item[(iii)]  Assume further that \(\phi\) is lower semicontinuous. Then $V_{K,\phi}$ is lower semicontinuous. Consequently, for such a function \(\phi\), $V_{K,\phi}$ is continuous 
if and only if $V_{K,\phi}^*\leq\phi$ on \(K\). 
\end{itemize}
\end{lem}

\section{Local H\"older continuity property and uniform density in capacity}\label{sec:hcpandudc}

\subsection{Moduli of continuity of extremal functions on complex coordinate spaces}
When \(E\subset\mathbb{C}^n\) is non-pluripolar, the 
modulus of continuity of \(L_E\) at \(a\in E\) is defined by 
$$\varpi_E(a, \delta) := \sup_{\|z-a\| \leq \delta} L_E(z),\quad
\delta\in[0,\infty).$$
The modulus of continuity of \(L_E\) on \(E\) is defined by 
$$\varpi_E(\delta) := \sup\{ \varpi_E(a,\delta): a\in E\},\quad
\delta\in[0,\infty).$$
$L_E$ is continuous at $a$ if and only if  
$\lim_{\delta\to 0^+} \varpi_E(a,\delta) =0$. Moreover, 
\[\label{eq:local-modulus-of-continuity}
|L_E(z) - L_E(w)| \leq \varpi_E(\|z-w\|), \quad z,w\in \mathbb{C}^n 
.\]
This can be proved as follows. For each \(z'\in\mathbb{C}^n\) and competitor \(u\) for \(L_E\), \(u(\cdot+z')-\varpi_E(\|z'\|)\leq L_E(\cdot)\) by definition of \(L_E\), thus \(L_E(z+z')-\varpi_E(\|z'\|)\leq L_E(z)\).

\subsection{Local H\"older continuity property}
We introduce the local H\"older continuity property and the weak local H\"older continuity property, which are related to the H\"older continuity of extremal functions on \(X\). The term H\"older continuity property is abbreviated as HCP.  
\begin{defn}\label{defn:local-mu-hcp} Let $q>0,\mu\in(0,1]$ be constants. Let \(F\) be a compact subset of \(\mathbb{C}^n\). $F$ is said to have 
\begin{itemize}
\item[(i)] $\mu$-HCP {at $a\in F$} if there exists a constant \(C>0\) satisfying
$$L_F(z) \leq C \delta^\mu,\quad0<\delta\leq1,\,\|z-a\|\leq\delta.$$
\item[(ii)] {\(\mu\)-HCP} if there exists a constant \(C>0\) satisfying 
$$L_F (z) \leq C \delta^\mu,\quad 0<\delta\leq1,\,z\in\bar{\mathbb{B}}(F,\delta).$$
\item[(iii)] local $\mu$-HCP of order $q$ at $a\in F$ if there exists a constant $C>0$ satisfying 
$$\varpi_{F\cap \bar{\mathbb{B}}(a,r)} (a,\delta) \leq \frac{C  \delta^\mu}{r^q}, \quad 0<\delta\leq 1,\, 0< r \leq 1.$$
\item[(iv)] local $\mu$-HCP of order $q$ if there exists a constant $C>0$ satisfying
$$\varpi_{F\cap \bar{\mathbb{B}}(a,r)} (a,\delta) \leq \frac{C  \delta^\mu}{r^q}, \quad a\in F,\,0<\delta\leq 1,\,0< r \leq 1.$$
\end{itemize}

Let \(K\) be a compact subset of \(X\). \(K\) is said to have
\begin{itemize}
\item[(v)] local $\mu$-HCP of order $q$ at $b\in K$ if for 
{each} holomorphic coordinate ball \((\Omega,\tau)\) centered at \(b\) of 
radius in \((0,\infty]\), there exists a constant \(C>0\) 
satisfying $$\varpi_{\tau(K\cap \Omega)\cap \bar{\mathbb{B}}(\mathbf{0},r)} (\mathbf{0},\delta) \leq \frac{C \delta^\mu}{r^q}, \quad 0<\delta\leq 1,\, 0< r \leq 1.
$$
\item[(vi)] local \(\mu\)-HCP of order \(q\) if for 
{each} finite open cover \(\{(O_s,g_s)\}_{s\in J}\) of \(X\) of holomorphic charts and 
{each} regular refined open cover \(\{(\hat{O}_s,g_s)\}\) of \(X\) (\(\hat{O}_s\) is relatively compact in \(O_s\) 
for each \(s\in J\)), there exists a constant \(C>0\) such that, for each \((b,s)\in K\times J\) of \(b\in K\cap \hat{O}_s\),
$$\varpi_{g_s(K\cap O_s)\cap \bar{\mathbb{B}}(g_s(b),r)}(g_s(b),\delta)
\leq \frac{C\delta^{\mu}}{r^q},\quad 0<\delta\leq1,\,0<r\leq 1.$$
Equivalently, \(K\) has local $\mu$-HCP of order $q$ if
for 
{each} finite open cover \(\{(\Omega_{j},\tau_j)\}_{j\in\Lambda}\) of \(X\) 
of holomorphic coordinate balls \(\tau_j(\Omega_j)=\mathbb{B}(\mathbf{0},5)\subset\mathbb{C}^n\) with the refined open cover \(\{(\tau_j^{-1}(\mathbb{B}(\mathbf{0},1)),\tau_j)\}_{j\in\Lambda}\) of \(X\),
there exists a constant $C'>0$ such that,
for each \((b,j)\in K\times\Lambda\) of \(b\in K\cap \tau_j^{-1}(\mathbb{B}(\mathbf{0},1))\),
$$\varpi_{\tau_j(K\cap\Omega_j)\cap\bar{\mathbb{B}}(\tau_j(b),r)
} (\tau_j(b),\delta)
\leq \frac{C'  \delta^\mu}{r^q}, \quad 0<\delta\leq 1,\, 0< r\leq1.$$
\item[(vii)] weak local $\mu$-HCP of order $q$ at $b\in K$ if for 
{some} holomorphic coordinate ball \((\Omega,\tau)\) centered at \(b\) of 
radius in \((0,\infty]\), there exists a constant \(C>0\) satisfying  
$$\varpi_{\tau(K\cap \Omega)\cap \bar{\mathbb{B}}(\mathbf{0},r)} (\mathbf{0},\delta) \leq \frac{C \delta^\mu}{r^q}, \quad 0<\delta\leq 1,\, 0< r \leq 1.
$$
\item[(viii)] weak local \(\mu\)-HCP of order \(q\) if for 
{some} finite open cover \(\{(O_s,g_s)\}_{s\in J}\) of \(X\) of holomorphic charts with 
{some} regular refined open cover \(\{(\hat{O}_s,g_s)\}\) of \(X\) (\(\hat{O}_s\) is  relatively compact in \(O_s\) 
for each \(s\in J\)), there exists a constant \(C>0\) such that, for each \((b,s)\in K\times J\) of \(b\in K\cap \hat{O}_s\),
$$\varpi_{g_s(K\cap O_s)\cap \bar{\mathbb{B}}(g_s(b),r)}(g_s(b),\delta)
\leq \frac{C\delta^{\mu}}{r^q},\quad 0<\delta\leq1,\,0<r\leq 1.$$
Equivalently, \(K\) has weak local $\mu$-HCP of order $q$ if
for 
{some} finite open cover \(\{(\Omega_{j},\tau_j)\}_{j\in\Lambda}\) of \(X\) 
of holomorphic coordinate balls of \(\tau_j(\Omega_j)=\mathbb{B}(\mathbf{0},5)
\) with the refined covering \(\{(\tau_j^{-1}(\mathbb{B}(\mathbf{0},1)),\tau_j)\}_{j\in\Lambda}\) of \(X\),
there exists a constant $C'>0$ such that,
for each \((b,j)\in K\times\Lambda\) of \(b\in K\cap \tau_j^{-1}(\mathbb{B}(\mathbf{0},1))\),
$$\varpi_{\tau_j(K\cap\Omega_j)\cap\bar{\mathbb{B}}(\tau_j(b),r)
} (\tau_j(b),\delta)
\leq \frac{C'  \delta^\mu}{r^q}, \quad 0<\delta\leq 1,\, 0< r\leq1.$$
\end{itemize}
\end{defn}
Local HCP implies HCP for a compact subset in \(\mathbb{C}^n\) and local HCP implies weak local HCP for a compact subset in \(X\).  
For all the items, checking the inequality for \(0<r\leq r_0\) for some \(r_0\in(0,1]\) is enough due to the monotonicity of extremal functions with respect to the set inclusion. For (i), (ii), (iii), (v) and (vii), checking the inequality for \(0<\delta\leq \delta_0\) for some \(\delta_0\in (0,1]\) is enough  by local boundedness of extremal functions of non-pluripolar sets. For (iv), (vi) and (viii), checking the inequality for \(0<\delta\leq \delta_0\) for some \(\delta_0\in (0,1]\) is enough due to inequality \eqref{eq:local-modulus-of-continuity}. Also by \eqref{eq:local-modulus-of-continuity}, 
$L_F$ is locally $\mu$-H\"older continuous on \(\mathbb{C}^n\) if and only if \(F\) has \(\mu\)-HCP. 

The equivalence in (vi) of this definition is explained as follows. For each \(s\in J\), since \(\overline{\hat{O}_s}\) is compact and is in \(O_s\), there exists positive integer \(N_s\) and \(N_s\) pairs \((b_{s,t},r_{s,t})\) in \(X\times (0,\infty)\) for \(1\leq t\leq N_s\) such that $$\overline{\hat{O}_s}\subset \cup_{1\leq t\leq N_s}\mathbb{B}(b_{s,t},r_{s,t}), \quad  
\cup_{1\leq t\leq N_s}\mathbb{B}(b_{s,t},5r_{s,t})\subset O_s.$$
$$(\Omega_{s,t},\tau_{s,t}):=(g_{s}^{-1}(\mathbb{B}(b_{s,t},5r_{s,t})),r_{s,t}^{-1}(g_{s}-g_s(b_{s,t})),\quad s\in J,\,1\leq t\leq N_s,$$ are holomorphic coordinate balls \(\tau_{s,t}(\Omega_{s,t})=\mathbb{B}(\mathbf{0},5)\) with the refined open cover \(\{(\tau_{s,t}^{-1}(\mathbb{B}(\mathbf{0},1)),\tau_{s,t})\}_{s\in J,1\leq t\leq N_s}\) of \(X\). Let \((b,s)\in K\times J\) be a pair of \(b\in K\cap\hat{O}_{s}\). Then there exists \(1\leq t\leq N_{s}\) such that \(b\in\tau_{s,t}^{-1}(\mathbb{B}(\mathbf{0},1))\). Equivalence follows from
$$\varpi_{g_{s}(K\cap O_{s})\cap\bar{\mathbb{B}}(g_s(b),r_{s,t}r)}(g_{s}(b),r_{s,t}\delta)
=\varpi_{\tau_{s,t}(K\cap \Omega_{s,t})\cap\bar{\mathbb{B}}(\tau_{s,t}(b),r)}(\tau_{s,t}(b),\delta)$$
for \(0<\delta<\infty\), \(0<r<5\). Equivalence in (viii) is also explained by this inequality.
\begin{remark}\label{rmk:local hcp and continuity}
In Definition~\ref{defn:local-mu-hcp}, if \(K\) has local \(\mu\)-HCP of order \(q\) at \(b\in K\), then letting \(\delta\to0^+\) in the inequality (v) implies that \(K\) is {\em locally \(L\)-regular} at \(b\) in the sense of \cite{Ah26}. We in \cite{Ah26} showed that local \(L\)-regularity of \(K\) at \(b\) gives the continuity of \(V_K\) at \(b\). Thus in this case, we know that \(V_K\) is continuous at \(b\).

Moreover, in Definition~\ref{defn:local-mu-hcp}, if \(K\) has weak local \(\mu\)-HCP of order \(q\) at \(b\in K\), then letting \(\delta\to0^+\) in the inequality (v) implies that \(K\) is {\em weakly locally \(L\)-regular} at \(b\) in the sense of  \cite{Ah26}. We in  \cite{Ah26} showed that weak local \(L\)-regularity of \(K\) at \(b\) gives the continuity of \(V_K\) at \(b\). Accordingly, even in this weaker condition, we also know the continuity of \(V_K\) at \(b\).

However, to guarantee the H\"older continuity of extremal functions on \(X\), weak local \(L\)-regularity and local \(L\)-regularity 
are not sufficient, and we need these stronger properties. 
\end{remark}

\begin{proof}[Proof of Theorem~\ref{thm:HCP implies localHCP in CN}] 

First, assume that \(F\) is convex and \(F\) has \(\mu\)-HCP. By the HCP condition, there exists a constant \(C>0\) such that
$$L_F(z)\leq Cd_{\mathbb{C}^n}(z,F)^{\mu},\quad d_{\mathbb{C}^n}(z,F)\leq1.$$
Since \(L_F^*\in\mathcal{L}\) and \(L_F^*\) is locally bounded from above, for the constant 
$$ C':=\max\{C,\sup_{d_{\mathbb{C}^n}(z,F)>1}\frac{L_F(z)}{d_{\mathbb{C}^n}(z,F)^{\mu}}\}>0,$$ 
we have 
\[\label{eq:HCP and HC}
L_F(z)\leq C'd_{\mathbb{C}^n}(z,F)^{\mu},\quad z\in\mathbb{C}^n.\]
Since \(F\) is convex, for each \(a\in F\) and \(0<r\leq1\), the affinely transformed set \(a+(F-a)/r\) contains \(F\). Since \(F\) is compact, there exists a small constant \(\epsilon\in(0,1)\) such that \(\bar{\mathbb{B}}(a,1/\epsilon)\) contains \(F\) for each \(a\in F\). Then we have 
$$\begin{aligned}
L_{F\cap\bar{\mathbb{B}}(a,r)}(z)&= L_{(a+\frac{F-a}{r\epsilon})\cap\bar{\mathbb{B}}(a,\frac{1}{\epsilon})}(a+\frac{z-a}{r\epsilon})\\
&\leq L_F(a+\frac{z-a}{r\epsilon})\\
&\leq {C'}{d_{\mathbb{C}^n}(a+\frac{z-a}{r\epsilon},F)^{\mu}}\\
&\leq C'\|\frac{z-a}{r\epsilon}\|^{\mu}\\
&=\frac{C'}{\epsilon^{\mu}}\frac{\|z-a\|^{\mu}}{r^{\mu}}, \quad z\in\mathbb{C}^n,
\; 0<r\leq1.\end{aligned}$$
Therefore, \(F\) has local \(\mu\)-HCP of order \(\mu\). 

For the second part, assume that \(b\in F\) is a star center of \(F\) and \(F\) has \(\nu\)-HCP at \(b\). By the HCP of \(F\) at \(b\), there exists a constant \(C_b>0\) such that
$$L_F(z)\leq C_b\|z-b\|^{\nu},\quad \|z-b\|\leq1.$$
Since \(L_F^*\in\mathcal{L}\) and \(L_F^*\) is locally bounded from above, for the constant 
$$ C_b':=\max\{C_b,\sup_{\|z-b\|>1}\frac{L_F(z)}{\|z-b\|^{\nu}}\}>0,$$ 
we have 
$$L_F(z)\leq C_b'd_{\mathbb{C}^n}(z,F)^{\nu},\quad z\in\mathbb{C}^n.$$
Since \(b\) is a star center of \(F\), for each \(t\in (0,1]\), the set \(b+(F-b)/t\) contains \(F\). Since \(F\) is compact, there exists \(\epsilon_b\in (0,1)\) such that \(\bar{\mathbb{B}}(b,1/\epsilon_b)\) contains \(F\). Then 
$$\begin{aligned}
L_{F\cap\bar{\mathbb{B}}(b,r)}(z)&= L_{(b+\frac{F-b}{r\epsilon_b})\cap\bar{\mathbb{B}}(b,\frac{1}{\epsilon_b})}(b+\frac{z-b}{r\epsilon_b})\\
&\leq L_F(b+\frac{z-b}{r\epsilon_b})\\
&\leq \frac{C_b'}{\epsilon_b^{\nu}}\frac{\|z-b\|^{\nu}}{r^\nu}, \quad z\in\mathbb{C}^n,\; 0<r\leq1.\end{aligned}$$
Therefore, \(F\) has local \(\nu\)-HCP of order \(\nu\) at \(b\). 
\end{proof}

\begin{remark}\label{rmk:convex HCP to local HCP}
By the inequality ~\eqref{eq:HCP and HC}, a compact subset \(F\) of \(\mathbb{C}^n\) has HCP if and only if \(L_F\) is H\"older continuous on \(\mathbb{C}^n\). Suppose a compact subset \(F\subset\mathbb{C}^n\) has \(\mu\)-HCP with a H\"older coefficient \(C\) as
$$L_F(z)\leq C\delta^{\mu},\quad 0<\delta\leq1,\;\|z-a\|\leq \delta,\;a\in F.$$
Let \(z'\in\mathbb{C}^n\) be a point in \(\mathbb{C}^n\) such that \(d_{\mathbb{C}^n}(z',F)>1\). Then there exists \(z''\in F\) such that 
 \(d_{\mathbb{C}^n}(z',F)=\|z'-z''\|\). Let \(N(z'):=
\lceil d_{\mathbb{C}^n}(z',F)\rceil>1\) be the smallest positive integer that is larger than or equal to \(d_{\mathbb{C}^n}(z',F)>1\). Then there exist \(N(z')+1\) points \(\xi_0:=z'\), \(\xi_1\), ..., \(\xi_{N(z')}:=z''\) in \(\mathbb{C}^n\) such that 
$$\|\xi_j-\xi_{j-1}\|= \frac{d_{\mathbb{C}^n}(z',F)}{N(z')}\leq1,\quad 1\leq j\leq N(z').$$
Then by the telescoping with the help of \eqref{eq:local-modulus-of-continuity}, we have 
$$\begin{aligned}
L_F(z')&=L_F(z')-L_F(z'')\\
&\leq \sum_{j=1}^{N(z')}|L_{F}(\xi_j)-L_F(\xi_{j-1})|\\
&\leq \sum_{j=1}^{N(z')}\varpi_F(1)\\
&\leq N(z')C\leq 2Cd_{\mathbb{C}^n}(z',F),\quad d_{\mathbb{C}^n}(z',F)>1.
\end{aligned}$$
Therefore, we get 
$$L_F(z)\leq 2C\max\{d_{\mathbb{C}^n}(z,F),d_{\mathbb{C}^n}(z,F)^{\mu}\},\quad z\in\mathbb{C}^n.$$
By this inequality, we know that in the proof of Theorem~\ref{thm:HCP implies localHCP in CN}, in the case of \(\mu=1\), \(2C\) can serve as \(C'\). 
\end{remark}

\subsection{Uniform density in capacity}
We define a uniform density in capacity of a compact subset of \(X\).
\begin{defn}
\label{defn:unifden} Take a finite open cover \(\{(\Omega_{j},\tau_j)\}_{j\in \Lambda}\) of \(X\) of holomorphic coordinate balls of \(\tau_j(\Omega_j)=\mathbb{B}(\mathbf{0},5)\subset\mathbb{C}^n\) with the refined cover \(\{(\tau_j^{-1}(\mathbb{B}(\mathbf{0},1)),\tau_j
)\}_{j\in \Lambda}\) of \(X\). 
We say that a compact subset $K$ of \(X\) has {a uniform density in capacity of order \(q\)}, if there exist constants $q>0$ and $\varkappa>0$ such that, for each \((a,j)\in K\times \Lambda\) with \(a\in K\cap \tau_j^{-1}(\mathbb{B}(\mathbf{0},1))\),
with the notation \(\bar{B}_j(a,r):=\tau_j^{-1}(\bar{\mathbb{B}}(\tau_j(a),r))\), the following inequality holds:
\[\label{eq:cap-density} 
\inf_{0<r\leq 1} \frac{Cap_{\omega}(K \cap \bar{B}_j(a,r))}{r^q} \geq \varkappa.
\]
We say that \(K\) has { a uniform density in capacity of order \(q\) at \(a\in K\)} if \eqref{eq:cap-density} holds for each \(j\in\Lambda\) with \(a\in \tau_j^{-1}(\mathbb{B}(\mathbf{0},1))\).
\end{defn}

\begin{remark}
\label{rmk:unifden}
Let us make the same settings as Definition~\ref{defn:unifden}. By the monotonicity of the capacity along set inclusion, 
checking the inequality \eqref{eq:cap-density} for \(0<r\leq r_0\) for some \(r_0\in(0,1]\) is enough for the uniform density in capacity.

Since the Euclidean distance induced from any coordinate chart and the Riemannian distance \(d_{\omega}\) for \((X,\omega)\) are comparable (bi-Lipschitz equivalent) on each compact subset contained in the domain of the coordinate chart,
\(K\) has a uniform density in capacity of order \(q\) if and only if there exist constants \(R_0>0\), \(q''=q>0\) and \(\varkappa''>0\) such that for each \(a\in K\),
\[\label{eq:cap-density-global}
	\inf_{0<R\leq R_0} \frac{Cap_{\omega}(K\cap\{p\in X:d_{\omega}(a,p)\leq R\})}{R^{q''}} \geq \varkappa''
    .\] 
Also, \(K\) has a uniform density in capacity of order \(q\) at \(a\in K\) if and only if \eqref{eq:cap-density-global} holds. 
Therefore, the property of a uniform density in capacity and the order \(q\) are independent of the choice of an open cover, while \(\chi\) depends on the cover.

Let \(\tilde{\omega}>0\) be another Hermitian metric on \(X\).
Since \(X\) is compact, \(\omega\) and \(\tilde{\omega}\) are comparable. In other words, for some constant \(c'\geq1\), \((c')^{-1}\omega\leq\tilde{\omega}\leq c'\omega\).
Then
$$(c')^{-n}Cap_{\omega}(\cdot)
\leq Cap_{\omega/c'}(\cdot)\leq Cap_{\tilde{\omega}}(\cdot)
\leq Cap_{c'\omega}(\cdot)\leq (c')^n Cap_{\omega}(\cdot).$$
Plus, \(d_{\omega}\) and \(d_{\tilde{\omega}}\) are bi-Lipschitz equivalent.
Therefore, the property of a uniform density in capacity and the order \(q\) are independent of the choice of an Hermitian metric, while \(\chi\) depends on the Hermitian metric. 
\end{remark}

\subsection{Relation of extremal functions and uniform density in capacity in the complex coordinate space} 
Suppose $K\subset X$ is a compact subset contained in a single holomorphic coordinate unit ball in the common-center coordinate \(5\)-radius ball  $(\Omega,\tau)$ with \(\tau(\Omega)=\mathbb{B}(\mathbf{0},5)\). 
Then \(\bar{\mathbb{B}}(\tau(K),1)\subset \mathbb{B}(\mathbf{0},2)\Subset\mathbb{B}(\mathbf{0},3)\Subset\mathbb{B}(\mathbf{0},5)\). 
By inequalities \eqref{eq:omcap-under-relcap} and \eqref{eq:relcap-under-omcap}, 
$K$ has a uniform density in capacity if and only if for some constants \(q_0>0\) and \(\chi_0>0\), for each \((a,r)\in K\times (0,1]\),
\[\label{eq:cap-density-loc}
\frac{Cap(\tau(K) \cap  \bar{\mathbb{B}}(\tau(a),r)\cap\mathbb{B}(\mathbf{0},2), \mathbb{B}(\mathbf{0},3))}{r^{q_0}}
\geq \varkappa_0.
\]
\(\mathbb{C}^n\)-version of \eqref{eq:cap-comparison}, \cite[Theorem~2.1]{AT84} or \cite[Theorem~2.7]{Ko05}, tells that inequality \eqref{eq:cap-density-loc} is equivalent to the existence of constants $A_1>0$ and $q_1>0$ satisfying 
\[\label{eq:cap-to-classicalSiciak}
\sup_{\bar{\mathbb{B}}(\mathbf{0},3)} L_{\tau(K)\cap  \bar{\mathbb{B}}(\tau(a),r)}^* \leq \frac{A_1}{r^{q_1}},\quad (a,r)\in K\times(0,1].\]
Specifically, if \eqref{eq:cap-density-loc} holds, then \eqref{eq:cap-to-classicalSiciak} holds with \(q_1=q_0\) and \(A_1=\varkappa_0^{-1}G(2)\) for some function \(G(t)\in(0,\infty)\) only depending on \(t\in(0,\infty)\). 
More generally, by \cite[Theorem~2.7]{Ko05}, for a compact non-pluripolar subset \(F\) of \(\mathbb{C}^n\), if \(0<r<R<\infty\) and \(F\subset \mathbb{B}(\mathbf{0},r)\), we have
\[\label{eq:capacity comparison in Cn}
\frac{2\pi}{Cap(F,\mathbb{B}(\mathbf{0},R))^{1/n}}\leq\sup_{\bar{\mathbb{B}}(\mathbf{0},R)}L_F^*\leq \frac{G(r)}{Cap(F,\mathbb{B}(\mathbf{0},R))}.
\] Conversely, if 
\eqref{eq:cap-to-classicalSiciak} holds, then 
\eqref{eq:cap-density-loc} holds with \(q_0=nq_1\) and \(\varkappa_0=(2\pi/A_1)^n\): for each \((a,r)\in K\times (0,1]\),
\[\label{eq:cap-density-sufficient-loc}
\frac{Cap(\tau(K)\cap \bar{\mathbb{B}}(\tau(a),r),\mathbb{B}(\mathbf{0},3))}{r^{nq_1}} \geq\frac{\int_{\mathbb{C}^n}(dd^c\frac{1}{2}\log{(\|z\|^2+1)})^n}{(A_1)^n}= \frac{(2\pi)^n}{(A_1)^n}.
\]

The following lemma (proved by Nguyen \cite[Lemma~3.3]{Ng24}) tells that local HCP implies a uniform density in density in \(\mathbb{C}^n\). 
\begin{lem}
\label{lem:cap-density} Let $K$ be a compact subset of $
\mathbb{C}^n$. Suppose \(K\) has local {$\mu$-HCP} of order $q$. Then \(K\) has a uniform density in capacity of order \(nq\). 
\end{lem}

\begin{proof} 
We can reduce to the case $K \subset \mathbb{B}(\mathbf{0},1)$. There exists a uniform constant $C>0$ such that
$$\varpi_{K \cap \bar{\mathbb{B}}(a,r)} (a,\delta) \leq \frac{C \delta^\mu}{r^q}, \quad 0<\delta \leq 4,\quad 0<r\leq1,\quad a\in K.$$
$$\sup_{\bar{\mathbb{B}}(\mathbf{0},3)} L_{K\cap \bar{\mathbb{B}}(a,r)} (z) \leq 4^{\mu}C/r^q, \quad 0<r\leq 1, \quad a\in K.$$
$$\frac{Cap(K\cap \bar{\mathbb{B}}(a,r),\mathbb{B}(\mathbf{0},3))}{r^{nq}} \geq \frac{(2\pi)^n}{(4^{\mu}C)^n}, \quad 0<r\leq 1, \quad a\in K.$$
The last inequality is due to \eqref{eq:cap-density-sufficient-loc}.
\end{proof}

\section{Extension of H\"older continuity of extremal functions on compact subsets}\label{sec:hcp}

\subsection{Demailly estimates} 
The following Demailly estimates in \cite{De94} help us to prove that the H\"older continuity of extremal function \(V_K\) on \(K\) for a compact subset \(K\) in \(X\) implies the H\"older continuity on \(X\). Since \(X\) is compact, there exists a constant \(R_{\omega}>0\) depending only on \((X,\omega)\) such that, for each function \(v\in PSH(X,\omega)\cap L^{\infty}(X)\) and real number \(0<\delta\leq R_{\omega}\), 
the regularization
\[\label{eq:phie}
\Psi_\delta v(x):=\frac{1}{\delta ^{2n}}\int_{\zeta\in T_x'X}
v({\text{exph}}_x(\zeta))\chi\Big(\frac{|\zeta|^2_{\omega(x)}}{\delta ^2}\Big)\,dV_{\omega(x)}(\zeta),\quad x\in X,\]
is a smooth function on \(X\).
Here, \(\text{exph}_x:T_x'X \to X\), given in \cite{De94}, is the formal holomorphic part of the exponential map ${\exp}_x: T_x'X \to X$ with respect to  Chern connection of \((X,\omega)\) on the holomorphic tangent space \(T_x'X\) at \(x\in X\). $dV_{\omega(x)}(\zeta):=(dd^c\|\zeta\|_{\omega(x)}^2)^n/(2^nn!)$ is the Lebesgue measure on $(T_x'X, \omega(x))$ where 
$$\omega(x)=\sum_{j,k=1}^n \frac{\sqrt{-1}h_{j\bar{k}}}{2}
dz^j\wedge d\bar{z}^k,\quad
\zeta=\sum_{j=1}^n\zeta^j\frac{\partial}{\partial z^j}
\in T_x'X,$$
$$\|\zeta\|_{\omega(x)}^2=\sum_{j,k=1}^n \frac{h_{j\bar{k}}}{2}\zeta^j\bar{\zeta}^k,$$
so that \(\int_{\|\zeta\|_{\omega(x)}\leq1}dV_{\omega(x)}(\zeta)
=\pi^n/n!\) (the volume of a unit ball in \(\mathbb{C}^n\)).
$\chi:[0,\infty)\rightarrow[0,\infty)$ is a fixed cut-off function supported on \([0,1]\) with \(\chi([0,1))\subset(0,\infty)\) and $\int_{\mathbb C^n}\chi(\Vert z\Vert^2)\,dV(z)=1$. For convention we define \(\Psi_0v\) as \(v\).

We import the following lemma from \cite[Lemma~4.1]{KN19}, which is a well-modified version of \cite[Lemma~1.12]{BD12}. 
Let \(\|\cdot\|_X\) be the supremum norm on \(X\) and \(USC(X)\) be the set of upper semicontinuous functions from \(X\) into \(\mathbb{R}\cup\{-\infty\}\).  
\begin{lem}\label{lem:kis} 
There exists a constant \(c_0>0\) depending only on \((X,\omega)\) such that, for a given $v\in PSH(X,\omega)\cap L^\infty(X)$, 
the functions 
\begin{equation}\label{kisleg}
USC(X)\ni v_{\delta, c}(x):= \inf _{ t\in (0,\delta ]}(\Psi_{t }v(x) + c_1t^2 +c_1t  -c \log\frac{t}{\delta }), \quad x\in X
\end{equation}
defined for \((c,\delta)\in  (0,\infty)\times (0,R_{\omega}]\)
satisfy the following two properties for constants \(c_1>0\) and \(\delta_1\in(0,R_\omega]\) depending only on \((X,\omega,\|v\|_{X})\): 
\begin{itemize}
\item[(i)]$[0,\delta_1]\ni t\mapsto\Psi_{t }v+c_1t^2$ is increasing, and 
the infimum in \eqref{kisleg} is attained for each \((x,c,\delta)\in  X\times (0,\infty)\times (0,\delta_1]\). 
\item[(ii)] $\omega+dd^c  v_{\delta,c }\geq -(c_0c+2c_1\delta)\,\omega$ for each \((c,\delta)\in  (0,\infty)\times (0,\delta_1]\).
\end{itemize}
\end{lem}
Note that \(c_1t^2+c_1t\) in \eqref{kisleg} was \(c_1(t^2-\delta^2)+c_1(t-\delta)\) in \cite[Lemma~4.1]{KN19}. The difference \(c_1\delta^2+c_1\delta\) does not depend on \(t\), comes out of the infimum and \(dd^c(c_1\delta^2+c_1\delta)=0\). 

\cite[Remark~4.6]{De94} proved that \(\psi_{\delta}v\) in  \eqref{eq:phie} is 
uniformly (on \(X\)) \(O(\delta^2)\)-close to 
\[\label{eq:convoution-compare}
v*\chi_{\delta}(p):=\int_{\bar{\mathbb{B}}(\mathbf{0},1)}v\circ z^{-1}(z(p)+\delta s)\chi(\|s\|^2)dV(s),\quad p\in z^{-1}(\Omega_{\delta}),\] 
where \(z:U\to\Omega\subset{\mathbb{C}}^n\) is a holomorphic coordinate map 
such that \(z(x)=\mathbf{0}\) and \((\partial/\partial z^j)_j\) is an orthonormal basis of \(T_x'X\). \(\Omega_\delta:=\{y\in \Omega:d(y,\partial\Omega)>\delta\}\). 
For this \(z\), 
\[\label{eq:av-loc}	\check{v}_\delta (p) := \frac{1}{V_{2n}(\bar{\mathbb{B}}(\mathbf{0},1)) \delta^{2n}} \int_{\bar{\mathbb{B}}(\mathbf{0},\delta)} v\circ z^{-1}(z(p)+y) dV_{2n}(y),\quad p\in z^{-1}(\Omega_{\delta})\]
can be compared to \(v*\chi_{\delta}\) on \(\Omega_{\delta}\) by
\cite[(5.11), (5.12)]{KN23}: there exists a \(\delta\)-independent constant $C>0$ such that, for each \(\delta>0\) with \(\Omega_{\delta}\neq\emptyset\),
\[\label{eq:cov-reg}
v*\chi_\delta-v\leq C (\check{v}_\delta-v),\quad
v_{\delta/2}-v\leq C (v*\chi_\delta-v).\]
\eqref{eq:convoution-compare}, \eqref{eq:cov-reg} and 
\cite[Remark~4.2]{GKZ08}(or \cite[Theorem~3.2]{Z20}) together give
\begin{cor}\label{cor:test-h} 
Let \(v\in PSH(X,\omega)\cap L^{\infty}(X)\). If constants \(C>0\), \(\delta_0>0\) and \(\mu\in(0,1]\) satisfy
$\Psi_\delta v - v\leq C\delta^\mu$ for \(\delta\in (0,\delta_0]\), 
then $v$ is $\mu$-H\"older continuous.
\end{cor}

\subsection{H\"older continuity of extremal functions}
We first prove that the H\"older continuity of \(V_K\) on \(K\)
implies the H\"older continuity of \(V_K\) on \(X\).
\begin{lem}\label{lem:property} Let \(K\) be a compact subset of \(X\) and \(\mu\in(0,1)\).
Suppose \(V_K\) is \(\mu\)-H\"older continuous on \(K\), i.e., there exist constants \(C>0\) and \(\delta_0>0\) 
satisfying \(V_K(x)\leq C\delta^{\mu}\) for \(d_{\omega}(x,K)\leq \delta\leq\delta_0\). Then \(V_K\) is \(\mu\)-H\"older continuous on \(X\).
\end{lem}

\begin{proof} 
The assumption implies $V_K^*=0$ on \(K\), so $v:=V_K$ is continuous by Corollary~\ref{cor:semicontinuity}. Consider $\Psi_t v$  of bounded $\omega$-psh function $v$ as \eqref{eq:phie}.
By Lemma~\ref{lem:kis}, there exist constants \(c_0>0\), \(c_1>0\), \(\delta_1>0\) 
such that for each \(c\in (0,\infty)\), 
$$v_{\delta, c} (x) = \inf_{t\in(0,\delta]} \left(\Psi_t v(x) + c_1 t^2+c_1 t - c \log \frac{t}{\delta}\right),\quad 0<\delta\leq R_{\omega}$$
are upper semicontinuous functions,
$[0,\delta_1]\ni t\mapsto\Psi_tv + c_1t^2$ increasing and
$$\omega + dd^c v_{\delta, c} \geq -(c_0 c +2c_1 \delta) \omega,\quad\delta\in(0,\delta_1],\; c\in(0,\infty).$$
We can shrink \(\delta_0\leq 1\) so that, \(\delta_0\leq\delta_1\) and for each \(\delta\in(0,\delta_0]\), $$f(\delta):=\frac{1}{c_0}-2\frac{c_1}{c_0}\delta^{1-\mu}>0$$ makes
$c_{\delta} := f(\delta)\delta^\mu $ satisfy $c_0 c_{\delta}+ 2c_1\delta = \delta^\mu$
and gives  
\[\label{eq:modify}
v_\delta := \frac{v_{\delta,c_{\delta}}}{1 + \delta^\mu} \in PSH(X,\omega).\]

Since \(\text{exph}\) is smooth in both variables, 
\(d_{\zeta}(\text{exph}_x)  (\mathbf{0})=\text{Id}_{T_x'X}\) and \(X\) is compact, there exists a Lipschitz constant \(D_{\omega}\) depending only on \(X,\omega,R_{\omega}\) such that 
$$d_{\omega}(x,\text{exph}_x(\zeta))\leq D_{\omega} \|\zeta\|_{\omega(x)},\quad x\in X,\;\zeta\in T_x'X,\;\|\zeta\|_{\omega(x)}\leq R_{\omega}.$$
We can shrink \(\delta_0\) once more so that $$v(x)\leq C\delta^{\mu},\quad d_{\omega}(x,K)\leq\delta\leq D_{\omega}\delta_0.$$ 
Then for the constant \(c_2:=C(D_{\omega})^{\mu}+2c_1\), for each \((a,\delta)\in K\times (0,\delta_0]\), 
\[\begin{aligned}\label{eq:bound1}
(1+ \delta^\mu) v_\delta(a)
&\leq  \Psi_\delta v(a) + c_1 \delta^2 + c_1\delta\\
&\leq \sup_{d_{\omega}(x,a)\leq D_{\omega}\delta}	v(x) + c_1\delta^2 + c_1 \delta\\
&\leq c_2 \delta^\mu.\end{aligned}\]
$v\geq 0$ makes $v_{\delta} \geq 0$, thus 
$$v_\delta(a)\leq c_2\delta^{\mu},\quad(a,\delta)\in K\times (0,\delta_0].$$
This result, the equation \eqref{eq:modify} and the definition of \(v\) together give
\[\label{eq:bound2}v_\delta \leq v + c_2\delta^\mu \quad\text{on } X,\quad\delta\in(0,\delta_0].\]
Using \eqref{eq:bound2} and the argument of \cite{DDGHKZ}, we can show H\"older continuity of \(v\) as follows. Fix $(x,\delta)\in X\times (0,\delta_0]$. By Lemma~\ref{lem:kis}, the 
infimum in the definition of $v_{\delta,c_\delta}(x)$ is attained at some $t_0 = t_0 (x,\delta)\in (0,\delta]$. By \eqref{eq:modify} and \eqref{eq:bound2}, we get
\[\notag\Psi_{t_0} v(x) + c_1t_0^2  +c_1t_0
- c_{\delta} \log \frac{t_0} {\delta} - (1+ \delta^\mu)v(x) \leq c_2\delta^\mu(1+ \delta^\mu).\]
From this inequality, for the constant \(c_3=c_3(K):=2c_2+\sup_X v\), we get
\[\notag\Psi_{t_0} v(x) + c_1t_0^2 
- c_{\delta} \log \frac{t_0} {\delta} - v(x) \leq c_3\delta^\mu.\]
Since \(t\mapsto \Psi_t v(x) +c_1t^2\) is increasing on \([0,\delta_1]\),
$$\Psi_{t_0} v(x) + c_1t_0^2
\geq (\Psi_t v(x) + c_1t^2)|_{t=0}=v(x).$$
Accordingly, we get
$c_{\delta} \log (\delta/t_0) \leq  c_3 \delta^\mu$.
Putting \(c_{\delta}=f(\delta)\delta^{\mu}\) to this inequality and then relating to the inequality 
\(f(\delta_0)\leq f(\delta)\), for \(\kappa := \exp (-c_3/{f(\delta_0)})\in (0,1)\), we get 
$t_0 \geq \kappa\delta$.
Since $\Psi_{\kappa\delta} v(x) + c_1(\kappa\delta)^2\leq \Psi_{t_0} v(x) + c_1t_0^2$, 
we get 
\[\begin{aligned} \notag
\Psi_{\kappa \delta} v(x)- v(x) 
&\leq\Psi_{\kappa \delta} v(x) +c_1 (\kappa\delta)^2 
- v(x) \\
&\leq \Psi_{t_0} v (x) +c_1t_0^2 
-v (x)  \\
&\leq v_{\delta,c_{\delta}} (x) - v (x)\\
&=\delta^\mu v_\delta(x)  + (v_\delta(x) -v(x)).\end{aligned}\]
\(\|v_\delta\|_{L^{\infty}(X)}=\|v_{\delta}\|_X\) for \({0<\delta\leq \delta_0}\) are bounded by a constant \(c_4=c_4(\|v\|_X,1/\kappa)\) as \(\delta/t_0(x,\delta)\leq 1/\kappa\).
Therefore, \eqref{eq:bound2} gives 
$\Psi_{\kappa \delta}v - v \leq (c_2+c_4) \delta^{\mu}$. In other words,
\[\label{eq:bound3}
\Psi_\delta v - v \leq (c_2+c_4)\kappa^{-\mu} \delta^{\mu}, \quad 0<\delta \leq \kappa\delta_0.\]
Thus we get the \(\mu\)-H\"older continuity of \(v=V_K\) by Corollary~\ref{cor:test-h}.
\end{proof}


Weighted extremal functions on \(X\) also have the property similar to Lemma~\ref{lem:property}.
\begin{prop} \label{prop:W} Let $K\subset X$ be a compact subset, \(\mu\in(0,1)\) and $\phi$ be a \(\mu\)-H\"older continuous real-valued function on \(X\). 
Suppose 
for some constants \(C,\delta_0>0\), 
\(V_{K,\phi}(x)\leq\phi(x)+ C\delta^{\mu}\) for \(d_{\omega}(x,K)\leq\delta\leq\delta_0\).
Then $V_{K,\phi}$ is \(\mu\)-H\"older continuous. 
\end{prop}

\begin{proof} $v:= V_{K,\phi}$ is continuous by Lemma~\ref{lem:weight-vs-unweight}-(iii) as \(v^*|_K\leq \phi|_K\). 
Let us use the same notations as in Lemma~\ref{lem:property}. Then equation \eqref{eq:bound1} becomes
$$(1+ \delta^\mu) v_\delta(a) \leq  \Psi_\delta v(a) + c_1 \delta^2 + c_1\delta
\leq \phi(a) + C' \delta^\mu,\quad (a,\delta)\in K\times (0,\delta_0],$$
where \(C'\) is the sum of \(C(D_{\omega})^{\mu}+2c_1\) and the H\"older coefficient of \(\phi\). Thus  $v_\delta(a) \leq \phi(a) + (\|v_\delta\|_{X} + C') \delta^\mu$. This and the definition of $v$ gives
$$v_\delta \leq v + (c_4+C')\delta^\mu,\quad \delta\in(0,\delta_0]$$ for \(\delta\)-independent constant \(c_4\) defined as in the proof of Lemma~\ref{lem:property}.
We have got \eqref{eq:bound2}. 
The same proof of Lemma~\ref{lem:property} gives the \(\mu\)-H\"older continuity of $v=V_{K,\phi}$.
\end{proof}

\section{Proofs of the main theorems}\label{sec:pfs of thms}

\subsection{Proof of the first theorem}
\begin{proof}[Proof of Theorem~\ref{thm:characterization-hcp}-(i)] The \(\mu\)-H\"older continuity of \(V_K\) at \(a\in K\) means $V_K(x) \leq C_a \delta^\mu$ for $d_{\omega}(x,a) \leq \delta \leq 1$ for some constant $C_a>0$. This implies $V_{K}^*(x) \leq C_a \delta^\mu$ for $d_{\omega}(x,a) \leq\delta<1$. By monotonicity, we can require \(r\leq1\).

Pick \(\psi\in C^{\infty}(\mathbb{R},[0,1])\) supported on \([-1,1]\) of \(\psi(0)=1\) and the 
chart \((U,f)\) of \(X\) with \(U\supset\bar{B}(a,r)=f^{-1}(\bar{\mathbb{B}}(f(a),r))\). They induce   
\(\chi_r\in C^{\infty}(X,[0,1])\) defined as
$$\chi_r(x):=\begin{cases}1-\psi(\|f(x)-f(a)\|/r),\quad x\in U,\\
1,\quad x\in X\setminus \bar{B}(a,r).\end{cases}$$ 
$\chi_r(a)=0$ and $\chi_r \equiv 1$ 
on \(X\setminus \bar{B}(a,r)\). Its \(C^1\)-norm and \(C^2\)-norm are bounded by
$$\|\chi_r\|_{C^1(X)} \leq c_1 /r, \quad \|\chi_r\|_{C^2(X)} \leq c_2 / r^2.$$
Here, constants $c_1>0$ and $c_2>0$ are independent of $a$ and $r$, but they depend on the choice of the holomorphic coordinate chart \((U,f)\). 
Accordingly, since \(\omega>0\), \(-dd^c\chi_r\geq-(c_2'/r^2)\omega\) for some constant \(c_2'\geq1\) depending only on 
\((U,f)\). Then for \(\varepsilon_r:={r^2}/{(2c_2')}\), the function \(-\varepsilon_r \chi_r\) is in \(PSH(X, \omega/2)\).

Let \(v\in PSH(X,\omega)\), \(v\leq1\) and \(v\leq 0\) on \(K\cap\bar{B}(a,r)\). By the choice of \(\chi_r\) and \(\varepsilon_r\), \(\varepsilon_r(v-\chi_r)\leq0\) on \(K\) and \(\varepsilon_r+1/2\leq 1\). These imply \(\varepsilon_r(v-\chi_r)\leq V_K\). Accordingly, 
$$\varepsilon_r  h^*_{K\cap  \bar{B}(a,r)} - \varepsilon_r \chi_r  \leq V_{K}^*.$$
This inequality and $V_K^*(a) =0=\chi_r(a)$ 
give $h_{K\cap  \bar{B}(a,r)}^*(a) =0$. 
Then by Lemma~\ref{lem:BTcap_globalandlocal}, $K \cap \bar{B}(a,r)$ is not pluripolar and $Cap_{\omega}(K \cap \bar{B}(a,r))>0$. By \eqref{eq:cap-comparison}, 
$$\sup_X V_{K\cap  \bar{B}(a,r)}^*\leq A/{Cap_{\omega}(K \cap \bar{B}(a,r))} < \infty.$$
For each \((x,\delta)\in X\times (0,\infty)\) of \(d_{\omega}(x,a)\leq\delta<1\), since $\chi_r (x) \leq c_1 \delta /r$, we have
$$\varepsilon_r  h_{K\cap  \bar{B}(a,r)}^* (x) \leq \varepsilon_r \chi_r(x) + C_a \delta^\mu \leq (\frac{c_1\varepsilon_r}{r} + C_a) \delta^\mu.
$$
By \eqref{eq:zero-one-function}, 
$V_{K\cap \bar{B}(a,r)}^*(a) \leq \sup_X V_{K\cap  \bar{B}(a,r)}^* h^*_{K\cap \bar{B}(a,r)}(a)$. Then $V_{K\cap \bar{B}(a,r)}^*(a)=0$. For the constants \(C_a\), \(c_1\) and \(c_2'\) which are independent of \(r\), we have 
\[\label{eq:H-norm}
V_{K\cap \bar{B}(a,r)}^* (x) 
\leq (\sup_X V_{K\cap \bar{B}(a,r)}^*) (\frac{c_1}{ r} + \frac{2c_2'C_a}{r^2})  \delta^\mu, \quad d_{\omega}(x, a)\leq \delta<1.\]
In other words, \(V_{K\cap \bar{B}(a,r)}\) is also \(\mu\)-H\"older continuous at \(a\).
\end{proof}

This proof of Theorem~\ref{thm:characterization-hcp}-(i) gives the following estimate. Let \(\mu\in(0,1]\).
\begin{cor} \label{cor:sharp-Holder-norm} 
Suppose \(V_K\) is continuous for compact \(K\subset X\). 
Let \(\{(\Omega_{j},\tau_j)\}_{j\in\Lambda}\) be a finite cover of \(X\) by holomorphic coordinate balls 
of \(\tau_j(\Omega_j)=\mathbb{B}(\mathbf{0},5)\subset\mathbb{C}^n\) and the refined cover \(\{(\tau_j^{-1}(\mathbb{B}(\mathbf{0},1)),\tau_j)\}_{j\in\Lambda}\) of \(X\). 
Denote \(\bar{B}_j(b,r):=\tau_j^{-1}(\bar{\mathbb{B}}(\tau_j(b),r))\). Let \(a\in K\).
\begin{itemize}
\item[(i)] If $V_K$ is $\mu$-H\"older continuous (resp. continuous at $a$) for \(\mu\in(0,1]\), then there exist constants $\delta_0, A'>0$ and \(r_0\in(0,1]\) such that, for each (a,j) (resp. for each \(j\)) of \(a\in K\cap \tau_j^{-1}(\mathbb{B}(\mathbf{0},1))\),
\[\label{eq:norm-loc}
	V_{K \cap  \bar{B}_j(a,r)} (x) \leq \frac{A'}{Cap_{\omega}(K \cap \bar{B}_j(a,r))} \frac{ \delta^\mu }{r^2}, \quad d_{\omega}(x, a) \leq \delta\leq\delta_0,\, 0<r\leq r_0.
\]
\item[(ii)] If there exist constants $\delta_0,A'>0$, $r_0\in(0,1]$ and \(\mu\in(0,1)\) such that \eqref{eq:norm-loc} holds for each \((a,j)\) (resp. for each j) of $a\in K\cap \tau_j^{-1}(\mathbb{B}(\mathbf{0},1))$, 
then $V_K$ is $\mu$-H\"older continuous (resp. continuous at a).
\end{itemize}
\end{cor}

\begin{proof} (i) follows from \eqref{eq:H-norm} and 
\eqref{eq:cap-comparison}. By the proof of Theorem~\ref{thm:characterization-hcp}-(i), \(A'\) depends only on the H\"older coefficient of \(V_K\) and the finite open cover \(\{(\Omega_{j},\tau_j)\}_{j\in\Lambda}\).

Showing the pointwise H\"older continuity of \(V_K\) in (ii) is obvious by monotonicity of extremal functions. We show the global H\"older continuity of \(V_K\) in (ii). 

\(X\) is a compact metric space with Riemannian distance \(d_{\omega}\). Thus by Lebesgue number lemma, some \(\eta\in (0,r_0]\) can be a Lebesgue number for the refined finite open cover \(\{(\tau_j^{-1}(\mathbb{B}(\mathbf{0},1)),\tau_j)\}_{j\in\Lambda}\) of \(X\).
By Remark~\ref{rmk:unifden},
there exist uniform positive constants \(0<l_1\leq1\leq l_2\) depending only on \(X,\omega,\{(\Omega_{j},\tau_j)\}_{j\in\Lambda}\) such that, 
for each \(j\in\Lambda\) and \(p,p'\in \tau^{-1}_j(\bar{\mathbb{B}}(\mathbf{0},4))\), 
$$l_1\|\tau_j(p)-\tau_j(p')\|\leq d_{\omega}(p,p')\leq l_2\|\tau_j(p)-\tau_j(p')\|.$$
Obviously, we can shrink \(\eta\) to satisfy \(\eta\leq l_1 r_0\). Again, since \(X\) is compact, there exists a finite open cover \(\{D_{\lambda}\}\) of \(X\), where each \(D_{\lambda}\)
has \(d_{\omega}\)-diameter less than \(\eta\).

Let $d_{\omega}(x, K) \leq \delta\leq \delta_0$. There exists $a\in K$ of $d_{\omega}(x, a) = d_{\omega}(x,K)$.
Then \(a\in D_{\lambda_a}\) for some \(\lambda_a\), \(D_{\lambda_a}\subset \tau_{j_a}^{-1}(\mathbb{B}(\mathbf{0},1))\) for some \(j_a\in\Lambda\). By the bi-Lipschitz equivalence of the distances, $$\bar{B}_{j_a}(a,\frac{\eta}{l_1})\supset O_a:=\{p\in X:d_{\omega}(a,p)\leq\eta\}\supset D_{\lambda_a}.$$
Whenever \(K\cap D_{\lambda}\neq\emptyset\), some \(b\in K\cap D_{\lambda}\) and \(j_b\in\Lambda\) with \(b\in\tau_{j_b}^{-1}(\mathbb{B}(\mathbf{0},1))\)
exist to satisfy \eqref{eq:norm-loc}. For some small \(r_b\in(0,r_0]\), \(\bar{B}_{j_b}(b,r_b)\subset D_{\lambda}\). By \eqref{eq:norm-loc}, \(K\cap \bar{B}_{j_b}(b,r_b)\) is non-pluripolar, \(Cap_{\omega}(K\cap \bar{B}_{j_b}(b,r_b))>0\), so \(Cap_{\omega}(K\cap D_{\lambda})>0\). Therefore,
$$\begin{aligned} Cap_{\omega}(K\cap \bar{B}_{j_a}(a,\frac{\eta}{l_1}))&\geq Cap_{\omega}(K\cap D_{\lambda_a})\\&\geq \min_{\lambda:K\cap D_{\lambda}\neq\emptyset} Cap_{\omega} (K\cap D_{\lambda})=:c\\&>0.\end{aligned}$$
Then the inclusion \(K\supset K\cap \bar{B}_{j_a}(a,\eta/l_1)\) and \eqref{eq:norm-loc} with \(\eta/l_1\leq r_0\) gives
$$V_K(x) \leq V_{K\cap \bar{B}_{j_a}(a,\frac{\eta}{l_1})} (x) \leq \frac{\tilde{A}}{c}(\frac{\eta}{l_1})^{-2}\delta^\mu.
$$
This inequality and Lemma~\ref{lem:property} prove the \(\mu\)-H\"older continuity of \(V_K\).
\end{proof}

With this corollary, we can prove Theorem~\ref{thm:characterization-hcp}-(ii).
\begin{proof}[Proof of Theorem~\ref{thm:characterization-hcp}-(ii)]
We want to show the H\"older continuity of \(V:=V_{K,\phi}\).
By taking the minimum of H\"older exponents of \(V_K\) and \(\phi\), we can assume that $V_K$ and $\phi$ are H\"older continuous with the same H\"older exponent \(0<\mu\leq1\). 

Let $b\in K$ and $\varepsilon'>0$. $\phi|_{\bar{B}(b,r)} \leq \phi(b) + \varepsilon'$ for small $r>0$. By \eqref{eq:monotonicity}, 
$$V\leq V_{K\cap \bar{B}(b,r),\phi}\leq V_{K\cap \bar{B}(b,r), \phi(b)+ \varepsilon'} = \phi(b)+\varepsilon' + V_{K\cap \bar{B}(b,r)}.$$
By Theorem~\ref{thm:characterization-hcp}-(i),
\(V_{K\cap \bar{B}(b,r)}^*(b)=0\). Then $V^*(b) \leq \phi(b)+\veps'$. $\varepsilon'\to0^+$ gives $V^* (b)\leq \phi (b)$. Accordingly, \(V^*|_K\leq \phi|_K\), \(V\) is continuous by Lemma~\ref{lem:weight-vs-unweight}. 
To prove the H\"older continuity of \(V\), by Proposition~\ref{prop:W}, it is enough to show that there exist constants $C>0$, $\mu'>0$ and $\delta_0>0$ satisfying 
$$V(x) \leq \phi(x) + C\delta^{\mu'}, \quad d_{\omega}(x, K) \leq \delta \leq \delta_0.$$

Since \(X\) is compact, it has a finite cover \(\{(\Omega_{j},\tau_j)\}_{j\in\Lambda}\)
of holomorphic balls \(\tau_j(\Omega_j)=\mathbb{B}(\mathbf{0},5)\subset\mathbb{C}^n\) with the refined cover \(\{(\tau_j^{-1}(\mathbb{B}(\mathbf{0},1)),\tau_j)\}_{j\in\Lambda}\).
Denote \(\bar{B}_j(b,r):=\tau_j^{-1}(\bar{\mathbb{B}}(\tau_j(b),r))\).
By Remark~\ref{rmk:unifden}, there exist positive constants \(0<l_1\leq1\leq l_2\) such that for each \(j\in\Lambda\) and \(p,p'\in \tau^{-1}_j(\bar{\mathbb{B}}(\mathbf{0},4))\), 
$$l_1\|\tau_j(p)-\tau_j(p')\|\leq{\rm dist}(p,p')\leq l_2\|\tau_j(p)-\tau_j(p')\|.$$

Set \(\delta_1:=(l_2)^{-1}\leq1\). Let \(x\in X\) with \(d_{\omega}(x, K) \leq \delta\leq\delta_1\). Pick $b_x\in K$ of $d_{\omega}(x, b_x) = d_{\omega}(x, K)$. Then \(b_x\in \tau_{j(b_x)}^{-1}(\mathbb{B}(\mathbf{0},1))\) for some \(j(b_x)\in\Lambda\). Let $r\in(0,\delta_1]$ be determined later. By H\"older continuity, $|\phi(w)- \phi(b_x)| \leq c_0 (l_2r)^\mu$ for each $w\in\bar{B}_{j(b_x)}(b_x,r)$, where \(c_0\) is a H\"older coefficient of \(\phi\). By the monotonicity \eqref{eq:monotonicity}, 
\[\label{eq:basic-inequality-holder1}\begin{aligned}
	V(x) \leq V_{K\cap \bar{B}_{j(b_x)}(b_x,r),\phi}(x)
    &\leq V_{K \cap  \bar{B}_{j(b_x)}(b_x,r), \phi(b_x) + c_0 l_2^{\mu}r^\mu}(x) \\
    &= \phi(b_x) + c_0 l_2^{\mu}r^\mu + V_{K\cap \bar{B}_{j(b_x)}(b_x,r)}(x) \\
    &\leq \phi(x) +  c_0\delta^\mu+ c_0 l_2^{\mu}r^\mu + V_{K\cap \bar{B}_{j(b_x)}(b_x,r)}(x).
\end{aligned}\]

Corollary~\ref{cor:sharp-Holder-norm}-(i) and a uniform density in capacity of $K$ give constants \(0<\delta_2\leq 1\), \(0<r_2<1\), \(0<A'\), \(0<\varkappa\) and \(0<q\) which are independent of \(b_x\) and \(j(b_x)\) such that, for each \((y,\delta,R)\in X\times (0,\infty)^2\) of \(d_{\omega}(y,b_x)\leq\delta\leq\delta_2\) and \(0<R\leq r_2\),
\[\label{eq:basic-inequality-holder2}
V_{K\cap \bar{B}_{j(b_x)}(b_x,R)}(y) \leq \frac{A'}{Cap_{\omega}(K \cap \bar{B}_{j(b_x)}(b_x,R))} \frac{\delta^\mu}{R^2} \leq \frac{A' \, \delta^\mu}{ \varkappa \,R^{q+2}}.\]

\((0,\infty)\ni s\mapsto\min\{s\mu,\mu-s(q+2)\}\) 
is maximized at \(s_0:={\mu}/{(\mu+q+2)}\in (0,1)\) with the maximum 
\(\mu':={\mu^2}/{(\mu+q+2)}<\mu\). 
Set \(\delta_0:=\min\{\delta_1,\delta_2,r_2\}\). Whenever \(d_{\omega}(z,K)\leq\delta\leq\delta_0\), there exists \(b_z\in K\) of \(d_{\omega}(z,b_z)=d_{\omega}(z,K)\leq\delta\leq\min\{\delta_1,\delta_2\}\), and 
\(r=R(\delta):=\delta^{s_0}\leq\delta\leq\delta_0\leq\min\{\delta_1,r_2\}\). Therefore, by \eqref{eq:basic-inequality-holder1} and \eqref{eq:basic-inequality-holder2}, for the constant \(C:=c_0+c_0 l_2^{\mu}+{A'}/{\varkappa}>0\), we have 
$$V(z)\leq \phi(z)+C\delta^{\mu'},
\quad d_{\omega}(z,K)\leq\delta\leq\delta_0.
$$
Then by Proposition~\ref{prop:W}, 
$V=V_{K,\phi}$ is H\"older continuous on $X$.
\end{proof}

\begin{remark}\label{rmk:exponent-W}The proof of Theorem~\ref{thm:characterization-hcp}-(ii) demonstrated that for a compact $K$ having a uniform density in capacity with \(q\) as in \eqref{eq:cap-density}, $\mu$-H\"older continuities of the extremal function \(V_K\) of \(K\) and 
\(\phi\) implies the $\mu'$-H\"older continuity of the weighted extremal function $V_{K,\phi}$, where $\mu'=\mu^2/(\mu+2+q)<\mu$. 
\end{remark}

\begin{remark}\label{rmk:converse-W}
Let \(\phi\) be a smooth real-valued function on \(X\). Then there exists a constant \(c>0\) such that \(\tilde{\omega}:=dd^c\phi+c\omega>0\) becomes a Hermitian metric. Assume \(V_{K,\phi/c}\) is \(\nu\)-H\"older continuous and \(K\) has a uniform density in capacity of order \(q\) for some \(\nu\in(0,1]\) and \(q>0\). 
By Remark~\ref{rmk:unifden},
\(K\) also has a uniform density in capacity of order \(q\) with respect to \(\tilde{\omega}\). 
Since \(V_{K,\phi/c}=c^{-1}(V_{\tilde{\omega};K}+\phi)\) is \(\nu\)-H\"older continuous, the function \(V_{\tilde{\omega};K}\) is also \(\nu\)-H\"older continuous. Thus \(V_{\tilde{\omega};K}\) is \(\nu\)-H\"older continuous. 
For \(\nu':=\nu^2/(\nu+2+q)\), by Theorem~\ref{thm:characterization-hcp}-(ii), \(V_{\tilde{\omega};K,-\phi}\) is \(\nu'\)-H\"older continuous. Therefore, the unweighted extremal function \(V_K=c^{-1}(V_{\tilde{\omega};K,-\phi}+\phi)\) of \(K\) is \(\nu'\)-H\"older continuous. 
\end{remark}

\subsection{Proof of the second theorem}
Finally, we prove Theorem~\ref{thm:characterization-hcp-hcp}. 
We first show the equivalence of (i), (ii) and (iii) in Theorem~\ref{thm:characterization-hcp-hcp}, where the implication (ii) \(\Rightarrow\) (iii) is obvious from the definitions. 
\begin{proof}[Proof of (iii) \(\Rightarrow\) (i) in Theorem~\ref{thm:characterization-hcp-hcp}] Assume \(K\) has weak local \(\mu\)-HCP of order \(q\) of \(K\) at \(a\in K\) for constants \(0<\mu\leq1\) and \(q>0\). Then there is a holomorphic coordinate ball \((\Omega,\tau)\) centered at \(a\) and a constant \(C>0\) satisfying  
$$\varpi_{\tau(K\cap\Omega)\cap\bar{\mathbb{B}}(\mathbf{0},r)}(\mathbf{0},\delta)\leq\frac{C\delta^{\mu}}{r^q},\quad 0<\delta\leq4,\; 0<r\leq 1.$$
In this inequality, for \(\delta=4\), we have 
\[\label{eq:basic-ineq-holder-loc1-pt}
\sup_{\bar{\mathbb{B}}(\mathbf{0},3)}L^*_{\tau(K\cap\Omega)\cap\bar{\mathbb{B}}(\mathbf{0},r)}\leq \frac{4^\mu C}{r^q},\quad 0<r\leq 1.
\]
Here, the range \(\delta\in(0,4]\) is available by \eqref{eq:local-modulus-of-continuity}. Set \(C':=4^{\mu}C\). \eqref{eq:cap-density-sufficient-loc} gives
$$\frac{Cap(\tau(K\cap\Omega)\cap\bar{\mathbb{B}}(\mathbf{0},r),\mathbb{B}(\mathbf{0},3))}{r^{nq}}\geq\frac{(2\pi)^n}{(C')^{n}},\quad 0<r\leq1.$$
There exists \(R\in(0,\infty]\) such that \(\tau(\Omega)=\mathbb{B}(\mathbf{0},R)\)). Let us define \(R_1:=\min\{3,R/2\}\) and \(R_2:=\min\{2,R_1/2\}\). Then we have 
$$\frac{Cap((\tau(K\cap\Omega)\cap \bar{\mathbb{B}}(\mathbf{0},r))\cap\mathbb{B}(\mathbf{0},R_2),\mathbb{B}(\mathbf{0},R_1))}{r^{nq}}\geq\frac{(2\pi)^n}{(C')^{n}},\quad 0<r\leq \frac{R_2}{2}.$$
Since \(0<R_2<R_1<R\), \eqref{eq:relcap-under-omcap} gives some constant \(C''>0\) satisfying
$$\frac{Cap_{\omega}(K\cap\tau^{-1}(\bar{\mathbb{B}}(\mathbf{0},r)))}{r^{nq}}\geq C'',\quad 0<r\leq \frac{R_2}{2}.$$
This implies \eqref{eq:cap-density-global} by
bi-Lipschitz equivalence of the Riemannian distance \(d_{\omega}\) and the Euclidean distance on a compact set in a coordinate chart. Therefore, \(K\) has a uniform density in capacity of order \(nq\) at \(a\).

To finish the proof, we want to show the H\"older continuity of \(V_K\) at \(a\in K\). There exists a bounded non-negative smooth strictly psh function of the form \(\rho_2(\cdot)=D\|\tau(\cdot)\|^2\) for some constant \(D>0\), on \(\tau^{-1}(\mathbb{B}(\mathbf{0},R_1))\), satisfying \(dd^c\rho_2\geq \omega\), \(\rho_2(a)=0\). By Lemma~\ref{lem:equivalence-notions}, 
for some constant \(M\in(0,\infty)\),
$$(V_{K\cap\tau^{-1}(\bar{\mathbb{B}}(\mathbf{0},R_2))}+\rho_2)\circ\tau^{-1}\leq M u_{\tau(K\cap\Omega)\cap\bar{\mathbb{B}}(\mathbf{0},R_2),\rho_2\circ\tau^{-1};\mathbb{B}(\mathbf{0},R_1)}\quad\text{on }\mathbb{B}(\mathbf{0},R_1).$$
The inequality \eqref{eq:re-ext-compare-c}, comparability \eqref{eq:compare-C} and inequality \eqref{eq:basic-ineq-holder-loc1-pt} yield that for each \(0<r\leq\min\{1,R_2\}\) and \(\|z\|\leq\delta\leq\min\{1,R_2\}\), as \(\tau(K\cap\Omega)\cap\bar{\mathbb{B}}(\mathbf{0},r)\subset\bar{\mathbb{B}}(\mathbf{0},R_2)\),
$$
\begin{aligned}
&u_{\tau(K\cap \Omega)\cap\bar{\mathbb{B}}(\mathbf{0},r),\rho_{2}\circ\tau^{-1};\mathbb{B}(\mathbf{0},R_1)}(z)\\
&\leq (1+\|\rho_{2}\|_{\tau^{-1}(\mathbb{B}(\mathbf{0},R_1))})u_{\tau(K\cap\Omega)\cap \bar{\mathbb{B}}(\mathbf{0},r);\mathbb{B}(\mathbf{0},R_1)}(z)+\sup_{\tau(K\cap \Omega)\cap\bar{\mathbb{B}}(\mathbf{0},r)}\rho_{2}\circ\tau^{-1}\\
&\leq\frac{1+\|\rho_{2}\|_{\tau^{-1}(\mathbb{B}(\mathbf{0},R_1))}}{\inf_{\partial\mathbb{B}(\mathbf{0},R_1)}L_{\tau(K\cap \Omega)\cap\bar{\mathbb{B}}(\mathbf{0},r)}}L_{\tau(K\cap \Omega)\cap\bar{\mathbb{B}}(\mathbf{0},r)}(z)+\sup_{\tau(K\cap \Omega)\cap\bar{\mathbb{B}}(\mathbf{0},r)}\rho_{2}\circ\tau^{-1}\\
&\leq\frac{1+\|\rho_{2}\|_{\tau^{-1}(\mathbb{B}(\mathbf{0},R_1))}}{\inf_{\partial\mathbb{B}(\mathbf{0},R_1)}L_{\bar{\mathbb{B}}(\mathbf{0},R_2)}}\frac{C\delta^{\mu}}{r^q}+Dr^2\\
&=\frac{1+DR_1^2}{\log{R_1}-\log{R_2}}\frac{C\delta^{\mu}}{r^q}+Dr^2.
\end{aligned}$$

The following function optimizes the exponent.
At \(t_0:=\mu/{(q+2)}\in(0,1)\),
$$(0,\infty)\ni t\mapsto \min\{\mu-qt,2t\}$$
attains its maximum \(\mu':={2\mu}/{(q+2)}<\mu\). Using this fact, for each \(\|z\|\leq\delta\leq(\min\{1,R_2\})^{1/t_0}\), choosing \(r(\delta):=\delta^{t_0}\leq\min\{1,R_2\}\), we get
$$u_{\tau(K\cap\Omega)\cap \bar{\mathbb{B}}(\mathbf{0},\delta^{t_0})),\rho_{2}\circ\tau^{-1};\mathbb{B}(\mathbf{0},R_1)}(z)\leq D'\delta^{\mu'},\quad D'=\frac{(1+DR_1^2)C}{\log{R_1}-\log{R_2}}+D.$$
Remember the setting \(\rho_2\geq 0\). Then
for each \(\|z\|\leq\delta\leq(\min\{1,R_2\})^{1/t_0}\), 
$$\begin{aligned}
V_{K}\circ\tau^{-1}(z)
&\leq (V_{K\cap \tau^{-1}(\bar{\mathbb{B}}(\mathbf{0},R_2))}+\rho_{2})\circ\tau^{-1}(z)\\
&\leq M u_{\tau(K\cap\Omega)\cap \bar{\mathbb{B}}(\mathbf{0},R_2),\rho_{2}\circ\tau^{-1};\mathbb{B}(\mathbf{0},R_1)}(z)\\
&\leq M u_{\tau(K\cap\Omega)\cap \bar{\mathbb{B}}(\mathbf{0},\delta^{t_0})),\rho_{2}\circ\tau^{-1};\mathbb{B}(\mathbf{0},R_1)}(z)\leq M D'\delta^{\mu'}.
\end{aligned}$$
Therefore, this inequality and equivalence of (\(d_{\omega}\), \((\tau\times\tau)^*d_{\mathbb{C}^n}\)) on \(\tau^{-1}(\bar{\mathbb{B}}(\mathbf{0},R_2))\) tell that \(V_K\) is
\(\mu'\)-H\"older continuous at the point \(\tau^{-1}(\mathbf{0})=a\in K\).\end{proof}

\begin{proof}[Proof of (i) \(\Rightarrow\) (ii) in Theorem~\ref{thm:characterization-hcp-hcp}]
Suppose there exist constants \(0<\nu\leq1\) and \(0<q_0\) such that \(V_K\) is \(\nu\)-H\"older continuous at \(a\in K\) and \(K\) has a uniform density in capacity of order \(q_0\) at \(a\). We want to the local HCP of \(K\) at \(a\). Fix a holomorphic coordinate ball \((\Omega_0,\tau_0)\) of radius \(R_0\in(0,\infty]\) centered at \(a\). 
Take \(R_{3}:=\min\{1,R_0/2\}\). 
Since \(K\) has a uniform density in capacity at \(a\) of order \(q_0\), there exists a constant \(\varkappa_0>0\) such that $$\frac{Cap_{\omega}(K\cap \tau_0^{-1}(\bar{\mathbb{B}}(\mathbf{0},r))}{r^{q_0}}\geq\varkappa_0,\quad 0<r\leq \frac{R_{3}}{2}.$$

By \eqref{eq:omcap-under-relcap},
there exists a bi-Lipschitz constant \(C_0>0\) such that 
$$Cap_{\omega}(\cdot\cap\tau_0^{-1}(\mathbb{B}(\mathbf{0},\frac{2R_{3}}{3})))
\leq
C_0Cap(\cdot\cap \tau_0^{-1}(\mathbb{B}(\mathbf{0},\frac{2R_{3}}{3})),\tau_0^{-1}(\mathbb{B}(\mathbf{0},R_{3}))).$$
Then for each \(0<r\leq R_{3}/2\), 
since \(R_{3}/2<2R_{3}/3\), we have 
$$Cap_{\omega}(K\cap\tau_0^{-1}(\bar{\mathbb{B}}(\mathbf{0},r)))
\leq C_0Cap(K\cap\tau_0^{-1}(\bar{\mathbb{B}}(\mathbf{0},r)),
\tau_0^{-1}(\mathbb{B}(\mathbf{0},R_{3}))),$$
$$\frac{Cap((K\cap\tau_0^{-1}(\bar{\mathbb{B}}(\mathbf{0},r))\cap\tau_0^{-1}(\mathbb{B}(\mathbf{0},\frac{2R_{3}}{3}))
,\tau_0^{-1}(\mathbb{B}(\mathbf{0},R_{3})))}{r^{q_0}}\geq \frac{\varkappa_0}{C_0},$$
$$\frac{Cap((\tau_0(K\cap\Omega_0)\cap\bar{\mathbb{B}}(\mathbf{0},r))\cap\mathbb{B}(\mathbf{0},\frac{2R_{3}}{3})
,\mathbb{B}(\mathbf{0},R_{3}))}{r^{q_0}}\geq \frac{\varkappa_0}{C_0}.$$
Since \(Cap(cE,cO)=Cap(E,O)\),
for each \(0<r\leq R_{3}/2\),
$$\frac{Cap(\frac{3}{R_{3}}(\tau_0(K\cap\Omega_0)\cap\bar{\mathbb{B}}(\mathbf{0},r))\cap\mathbb{B}(\mathbf{0},2)
,\mathbb{B}(\mathbf{0},3))}{r^{q_0}}
\geq 
\frac{\varkappa_0}{C_0}.$$
By \eqref{eq:cap-to-classicalSiciak},
there exists a constant \(C_1>0\) such that,
for each \(0<r\leq R_3/2\), 
\[\label{eq:cap-to-classicalSiciak-in-the-second-pt}
\sup_{\bar{\mathbb{B}}(\mathbf{0},R_3)}L_{
(\tau_0(K\cap\Omega_0)\cap\bar{\mathbb{B}}(\mathbf{0},r))}^*
=\sup_{\bar{\mathbb{B}}(\mathbf{0},3)}L_{\frac{3}{R_{3}}
(\tau_0(K\cap\Omega_0)\cap\bar{\mathbb{B}}(\mathbf{0},r))}^*\leq\frac{C_1}{r^{q_0}}.
\]

Since \(V_K\) is \(\nu\)-H\"older continuous at \(a\)
and \(K\) has a uniform density in capacity at \(a\) of order \(q_0\),
Corollary~\ref{cor:sharp-Holder-norm}-(i) 
guarantees a constant \(A'>0\) satisfying 
$$V_{K\cap \tau_0^{-1}(\bar{\mathbb{B}}(\mathbf{0},r))}(x)\leq \frac{A'}{\varkappa_0}\frac{\delta^{\nu}}{r^{q_0+2}}, \quad d_{\omega}(x,a)\leq\delta\leq 1,\;0<r\leq \frac{R_3}{2}.$$
Since \(R_3<R_0\), there exists a bounded non-negative smooth strictly psh function \(\rho_1\) on \(\tau_0^{-1}(\mathbb{B}(\mathbf{0},R_3))\) of the form \(\rho_1(\cdot)=d\|\tau_0(\cdot)\|^2\) for a constant \(d>0\)
such that $$\|\rho_1\|_{\tau_0^{-1}(\mathbb{B}(\mathbf{0},R_3))}<\infty,\quad \rho_1(a)=0,\quad dd^c\rho_1\leq\omega,\quad \liminf_{x\to\partial(\tau_0^{-1}(\mathbb{B}(\mathbf{0},R_3)))}\rho_1(x)>0.$$

There exists a Lipschitz constant \(l>1\) such that \(d_{\omega}\) is dominated by \(l\times((\tau_0\times\tau_0)^*d_{\mathbb{C}^n})\) on \(\tau_0^{-1}(\bar{\mathbb{B}}(\mathbf{0},R_3))\). The 
comparability \eqref{eq:compare-C}, inequality \eqref{eq:re-ext-compare-c} and Lemma~\ref{lem:equivalence-notions} in turn give that,
for each \(\|z\|\leq\delta\leq R_3/l\) and \(0<r\leq R_3/2\),
with $$m:=\frac{\liminf_{x\to\partial(\tau_0^{-1}(\mathbb{B}(\mathbf{0},R_3)))}\rho_1(x)}{1+{\|\rho_1\|}_{\tau_0^{-1}(\mathbb{B}(\mathbf{0},R_3))}}=\frac{R_3^2d}{1+R_3^2d}\quad\text{and}$$
$$c(r):=\sup_{\mathbb{B}(\mathbf{0},R_3)}L_{\tau_0(K\cap\Omega_0)\cap \bar{\mathbb{B}}(\mathbf{0},r)}\,,$$
we have the following chain of inequalities: 
$$\begin{aligned}
L_{\tau_0(K\cap\Omega)\cap\bar{\mathbb{B}}(\mathbf{0},r))}(z)
&\leq c(r) u_{\tau_0(K\cap\Omega_0)\cap\bar{\mathbb{B}}(\mathbf{0},r);\mathbb{B}(\mathbf{0},R_3)}(z)\\
&\leq c(r) (u_{\tau_0(K\cap\Omega_0)\cap \bar{\mathbb{B}}(\mathbf{0},r),\rho_{1};\mathbb{B}(\mathbf{0},R_3)}(z)
-\inf_{\mathbb{B}(\mathbf{0},R_3)}\rho_{1})\\
&=c(r) u_{\tau_0(K\cap\Omega_0)\cap \bar{\mathbb{B}}(\mathbf{0},r),\rho_{1};\mathbb{B}(\mathbf{0},R_3)}(z)\\
&\leq \frac{c(r)}{m} (V_{K\cap\tau_0^{-1}(\bar{\mathbb{B}}(\mathbf{0},r))}(\tau_0^{-1}(z))+\rho_{1}(\tau_0^{-1}(z)))\\
&\leq \frac{c(r)}{m}(\frac{A'(l\delta)^{\nu}}{\varkappa_0 r^{q_0+2}}+d \delta^2)\\
&\leq \frac{1+R_3^2 d}{R_3^2 d}C_1(\frac{A'}{\varkappa_0}\frac{l^{\nu}}{r^{2q_0+2}}+\frac{d}{r^{q_0}})\delta^{\nu}\\
&\leq\frac{1+R_3^2 d}{R_3^2 d}C_1(\frac{A'}{\varkappa_0}l^{\nu}+d)
\frac{\delta^{\nu}}{r^{2q_0+2}}.
\end{aligned}$$
In this chain of inequalities, \eqref{eq:cap-to-classicalSiciak-in-the-second-pt} is utilized to dominate \(c(r)\). 
This shows that \(K\) has local \(\nu\)-H\"older continuity property at \(a\in K\) of order \(2q_0+2\). 
\end{proof}

We have shown the equivalence (i), (ii) and (iii) in Theorem~\ref{thm:characterization-hcp-hcp}. The proof of the equivalence (iv), (v) and (vi) in Theorem~\ref{thm:characterization-hcp-hcp} uses the proof of the equivalence (i), (ii) and (iii), with additional uniform controls for global properties. Again, the implication (v) \(\Rightarrow\) (vi) in this theorem is obvious from the definitions.
\begin{proof}[Proof of (vi) \(\Rightarrow\) (iv) in Theorem~\ref{thm:characterization-hcp-hcp}] Assume that there exist constants \(0<\mu\leq1\) and \(q>0\) such that \(K\) has weak local \(\mu\)-HCP of order \(q\) with respect to some finite cover \(\{(\Omega_j,\tau_j)\}_{j\in\Lambda}\) of \(X\) of holomorphic coordinate balls \(\tau_j(\Omega_j)=\mathbb{B}(\mathbf{0},5)\) with its refined cover \(\{(\tau_j^{-1}(\mathbb{B}(\mathbf{0},1)),\tau_j)\}_{j\in\Lambda}\) of \(X\). We want to show  a uniform density in capacity of \(K\) and the H\"older continuity of \(V_K\). 
Write \(\tau_j^{-1}(\bar{\mathbb{B}}(\tau_j(b),r))=:\bar{B}_j(b,r)\) and \(\tau_j^{-1}(\mathbb{B}(\tau_j(b),r))=:B_j(b,r)\). For each \(j\in\Lambda\), let \(b_j\in X\) be the point of \(\tau_j(b_j)=\mathbf{0}\).
There exists a constant \(C>0\) such that, for each \((b,j)\in K\times\Lambda\) of \(b\in K\cap B_j(b_j,1)=K\cap \tau_j^{-1}(\mathbb{B}(\mathbf{0},1))\),
$$\varpi_{\tau_j(K\cap \bar{B}_j(b,r))} (\tau_j(b),\delta)
\leq \frac{C  \delta^\mu}{r^q}, \quad 0<\delta\leq 5,\, 0< r \leq 1.$$
Here, \(\delta\in(0,5]\) instead of \(\delta\in(0,1]\) is possible due to \eqref{eq:local-modulus-of-continuity}.
By Remark~\ref{rmk:unifden}, there exist Lipschitz constants \(0<l_1\leq1\leq l_2\) such that for each \(j\in \Lambda\) and \(p,p'\in\tau_j^{-1}(\bar{\mathbb{B}}(\mathbf{0},4))\), 
$$l_1\|\tau_j(p)-\tau_j(p')\|\leq d_{\omega}(p,p')\leq l_2\|\tau_j(p)-\tau_j(p')\|.$$
Also, there exist uniform constants \(d,D>0\) such that, for each pair \((p,j)\in X\times \lambda\) with \(p\in \tau_j^{-1}(\mathbb{B}(\mathbf{0},1))\), $$\rho_{1,p,j}(\cdot):=d\|\tau_j(\cdot)-\tau_j(p)\|^2,\quad \rho_{2,p,j}(\cdot):=D\|\tau_j(\cdot)-\tau_j(p)\|^2$$ are non-negative bounded smooth strictly psh functions on \(B_j(p,3)\) satisfying $$dd^c\rho_{1,p,j}\leq\omega\leq dd^c\rho_{2,p,j},\quad \rho_{1,p,j}(p)=0=\rho_{2,p,j}(p),$$
$$\liminf_{x\to\partial B_j(p,3)}\rho_{1,p,j}(x)>0.$$

The weak local HCP tells, for each \((b,j)\in\ K\times\Lambda\) of \(b\in K\cap \tau_j^{-1}(\mathbb{B}(\mathbf{0},1))\),
\[\label{eq:basic-ineq-holder-loc1}L_{\tau_j(K\cap \bar{B}_{j}(b,r))}(z)\leq\frac{C\delta^{\mu}}{r^q}, \quad \|z-\tau_j(b)\|\leq\delta\leq5,0<r\leq1.\]
Then for \(\delta=5\) and \(C':=5^{\mu}C\), for each
\((b,j)\in\ K\times\Lambda\) of \(b\in K\cap \tau_j^{-1}(\mathbb{B}(\mathbf{0},1))\),
$$\sup_{\bar{\mathbb{B}}(\mathbf{0},3)}L_{\tau_j(K\cap\bar{B}_j(b,r))}^*\leq \frac{C'}{r^q}, \quad 0<r\leq1.$$
By the inequality \eqref{eq:cap-density-sufficient-loc},
for each \((b,j)\in\ K\times\Lambda\) of \(b\in K\cap \tau_j^{-1}(\mathbb{B}(\mathbf{0},1))\),
$$\frac{Cap(\tau_j(K\cap \bar{B}_j(b,r))\cap\mathbb{B}(\mathbf{0},2),\mathbb{B}(\mathbf{0},3))}{r^{nq}} \geq \frac{(2\pi)^n}{(C')^n}, \quad 0<r\leq1,$$
$$\frac{Cap((K\cap \bar{B}_j(b,r))\cap B_j(b_j,2),B_j(b_j,3))}{r^{nq}} \geq \frac{(2\pi)^n}{(C')^n}, \quad 0<r\leq1,$$ 
$$\frac{Cap(\{p\in K:d_{\omega}(b,p)\leq l_2r\}\cap B_j(b_j,2),B_j(b_j,3))}{r^{nq}} \geq \frac{(2\pi)^n}{(C')^n},\quad 0<r\leq1,$$
as \(\{p\in X:d_{\omega}(b,p)\leq l_2r\}\supset \bar{B}_j(b,r)\). Thus for each \(b\in K\), 
$$\frac{\sum_{i\in\Lambda} Cap(\{p\in K:d_{\omega}(b,p)\leq l_2r\}\cap B_i(b_i,2),B_i(b_i,3))}{r^{nq}} \geq \frac{(2\pi)^n}{(C')^n},\quad 0<r\leq1,$$
as \(b\in K\cap \tau_{j_b}^{-1}(\mathbb{B}(\mathbf{0},1))\) for some \(j_b\in\Lambda\). Bi-Lipschitz equivalence of two capacities proved in Lemma~\ref{lem:BTcap_globalandlocal} guarantees a constant \(C''>0\) satisfying  
$$\frac{Cap_{\omega}(\{p\in K:d_{\omega}(b,p)\leq R\})}{R^{nq}} \geq C'',\quad 0<R\leq l_2, \quad b\in K.$$
This is \eqref{eq:cap-density-global} in Remark~\ref{rmk:unifden}, which is a uniform density in capacity of \(K\) with \(nq\).

Now we want to show the H\"older continuity of \(V_K\).
Fix an arbitrary pair \((p,s)\) in \(K\times\Lambda\) with \(p\in K\cap \tau_s^{-1}(\mathbb{B}(\mathbf{0},1))\). Lemma~\ref{lem:equivalence-notions} guarantees that for $$M_s(p):=\|V_{K\cap\bar{B}_s(p,2)}^*\|_{B_s(p,3)}+\|\rho_{2,p,s}\|_{B_s(p,3)}+1<\infty,$$
$$(V_{K\cap\bar{B}_s(p,2)}+\rho_{2,p,s})\circ\tau_s^{-1}\leq M_s(p) u_{\tau_s(K\cap \bar{B}_s(p,2)),\rho_{2,p,s}\circ\tau_s^{-1};\mathbb{B}(\tau_s(p),3)}\text{ on }\mathbb{B}(\tau_s(p),3).$$
The inequality \eqref{eq:re-ext-compare-c}, comparability \eqref{eq:compare-C} and inequality \eqref{eq:basic-ineq-holder-loc1} yield that for each \(r\in(0,1]\) and for each \(\|z-\tau_s(p)\|\leq\delta\leq1\), as \(\tau_s(K\cap\bar{B}_s(p,r))\subset\bar{\mathbb{B}}(\tau_s(p),1)\),
$$
\begin{aligned}
&u_{\tau_s(K\cap \bar{B}_s(p,r)),\rho_{2,p,s}\circ\tau_s^{-1};\mathbb{B}(\tau_s(p),3)}(z)\\
&\leq (1+\|\rho_{2,p,s}\|_{B_s(p,3)})u_{\tau_s(K\cap \bar{B}_s(p,r));\mathbb{B}(\tau_s(p),3)}(z)+\sup_{\tau_s(K\cap \bar{B}_s(p,r))}\rho_{2,p,s}\circ\tau_s^{-1}\\
&\leq\frac{1+\|\rho_{2,p,s}\|_{B_s(p,3)}}{\inf_{\partial\mathbb{B}(\tau_s(p),3)}L_{\tau_s(K\cap \bar{B}_s(p,r))}}L_{\tau_s(K\cap \bar{B}_s(p,r))}(z)+\sup_{\tau_s(K\cap \bar{B}_s(p,r))}\rho_{2,p,s}\circ\tau_s^{-1}\\
&\leq\frac{1+\|\rho_{2,p,s}\|_{B_s(p,3)}}{\inf_{\partial\mathbb{B}(\tau_s(p),3)}L_{\bar{\mathbb{B}}(\tau_s(p),1)}}\frac{C\delta^{\mu}}{r^q}+Dr^2\\
&=\frac{1+9D}{\log{3}}\frac{C\delta^{\mu}}{r^q}+Dr^2.
\end{aligned}
$$\(\min\{\mu-qt,2t\}\) for \(t\in(0,\infty)\) is maximized at \(t_0:=\frac{\mu}{q+2}\in(0,1)\) as \(2t_0=\mu-qt_0=\frac{2\mu}{q+2}=:\mu'<\mu\). Accordingly, for each \(\|z-\tau_s(p)\|\leq\delta\leq1\) with \(r(\delta):=\delta^{t_0}\leq1\),
$$u_{\tau_s(K\cap \bar{B}_s(p,\delta^{t_0})),\rho_{2,p,s}\circ\tau_s^{-1};\mathbb{B}(\tau_s(p),3)}(z)\leq D'\delta^{\mu'},\quad D'=\frac{(1+9D)C}{\log{3}}+D.$$

Then for each \(\|z-\tau_s(p)\|\leq\delta\leq1\), since \(\rho_{2,p,s}\geq0\),
$$\begin{aligned}
V_{K}\circ\tau_s^{-1}(z)
&\leq (V_{K\cap \bar{B}_s(p,2)}+\rho_{2,p,s})\circ\tau_s^{-1}(z)\\
&\leq M_s(p) u_{\tau_s(K\cap \bar{B}_s(p,2)),\rho_{2,p,s}\circ\tau_s^{-1};\mathbb{B}(\tau_s(p),3)}(z)\\
&\leq M_s(p) u_{\tau_s(K\cap \bar{B}_s(p,\delta^{t_0})),\rho_{2,p,s}\circ\tau_s^{-1};\mathbb{B}(\tau_s(p),3)}(z)\leq M_s(p) D'\delta^{\mu'}.
\end{aligned}$$
For \(p\in \tau_s^{-1}(\mathbb{B}(\mathbf{0},1))= B_s(b_s,1)\), we have the set inclusion \(\bar{B}_s(p,2)\supset\bar{B}_s(b_s,1)=\tau_s^{-1}(\mathbb{\bar{B}}(\mathbf{0},1))\). Accordingly, 
$$\begin{aligned}
M_s(p)&\leq \|V_{K\cap\bar{B}_s(b_s,1)}^*\|_{B_s(b_s,4)}+9D+1\\
&\leq \sup_{j\in\Lambda:K\cap B_j(b_j,1) \neq\emptyset}\|V_{K\cap\bar{B}_s(b_s,1)}^*\|_{B_s(b_s,4)}+9D+1
=:D''.\end{aligned}$$
\(D''\) is a constant independent of the pair \((p,s)\in K\times\Lambda\) of \(p\in K\cap \tau_s^{-1}(\mathbb{B}(\mathbf{0},1))\).
Let \(0<\eta\leq1\) be a Lebesgue number for \(\{(\tau_j^{-1}(\mathbb{B}(\mathbf{0},1)),\tau_j)\}\) of \((X,d_{\omega})\).
Pick \(\delta_0:=l_1\eta/2\). 
Let \((x,\delta)\in X\times [0,\infty)\) of \(d_{\omega}(x,K)\leq \delta\leq\delta_0\). There exists \(p_x\in K\) of \(d_{\omega}(x,K)=d_{\omega}(x,p_x)\). Since \(\delta_0<\eta\), there exists an index \(s_x\in\Lambda\) of \(\{x,p_x\}\subset \tau_{s_x}^{-1}(\mathbb{B}(\mathbf{0},1))=B_{s_x}(b_{s_x},1)\). Then \(\|\tau_{s_x}(x)-\tau_{s_x}(p_x)\|\leq \delta/l_1\leq\delta_0/l_1\leq1\),
$$V_K(x)=V_K\circ \tau_{s_x}^{-1}(\tau_{s_x}(x))\leq M_{s_x}(p_x)D'(\delta/l_1)^{\mu'}\leq D'' D'(l_1)^{-\mu'}\delta^{\mu'}.$$
Therefore, \(V_K\) is \(\mu'\)-H\"older continuous by Lemma~\ref{lem:property}.
\end{proof}

\begin{proof}[Proof of (iv) \(\Rightarrow\) (v) in Theorem~\ref{thm:characterization-hcp-hcp}]
Assume that 
\(V_K\) is \(\nu\)-H\"older continuous and \(K\) has a uniform density in capacity for constants \(0<\nu\leq1\) and \(q_0>0\). 
We will show that \(K\) has the local HCP. 
Fix finite cover \(\{(\Omega_j,\tau_j)\}_{j\in\Lambda}\) of \(X\) of holomorphic coordinate balls \(\tau_j(\Omega_j)=\mathbb{B}(\mathbf{0},5)\) with its refined cover \(\{(\tau_j^{-1}(\mathbb{B}(\mathbf{0},1)),\tau_j)\}_{j\in\Lambda}\) of \(X\). 
Let \(\varkappa_0\) be a constant for this cover for the uniform density in capacity \eqref{eq:cap-density} of \(K\). 
We use the settings and notations same as those of the proof of (vi) \(\Rightarrow\) (iv) in Theorem~\ref{thm:characterization-hcp-hcp}: \(\bar{B}_j(b,r)\), \(B_j(b,r)\), \(b_j\), \(0<l_1\leq1\leq l_2\), \(d>0\), \(D>0\), \(\rho_{1,p,j}(\cdot):=d\|\tau_j(\cdot)-\tau_j(p)\|^2\), \(\rho_{2,p,j}(\cdot):=D\|\tau_j(\cdot)-\tau_j(p)\|^2\).

Let \(C_0\) be a bi-Lipschitz constant as in \eqref{eq:omcap-under-relcap} such that $$Cap_{\omega}(\cdot\cap B_j(b_j,2))
\leq
C_0Cap(\cdot\cap B_{j}(b_j,2),B_j(b_j,3)),\quad j\in\Lambda.$$
For each \((b,j)\in K\times\Lambda\) of \(b\in K\cap\tau_j^{-1}(\mathbb{B}(\mathbf{0},1))\) and \(r\in(0,1]\), 
$$Cap_{\omega}((K\cap\bar{B}_j(b,r))\leq C_0\times Cap((K\cap\bar{B}_j(b,r))\cap B_{j}(b_j,2),B_j(b_j,3)).$$
Accordingly, for each \((b,j)\in K\times\Lambda\) of \(b\in K\cap\tau_j^{-1}(\mathbb{B}(\mathbf{0},1))\),
$$\frac{Cap((K\cap\bar{B}_j(b,r))\cap B_{j}(b_j,2)
,B_j(b_j,3))}{r^{q_0}}\geq \frac{\varkappa_0}{C_0}, \quad 0<r\leq1.$$
By \eqref{eq:cap-to-classicalSiciak}, there exists a constant \(C_1>0\) such that,
\[\label{eq:cap-to-classicalSiciak-in-the-second}\sup_{\bar{\mathbb{B}}(\mathbf{0},3)}L_{\tau_j(K\cap\bar{B}_j(b,r))}^*\leq\frac{C_1}{r^{q_0}}, \quad 0<r\leq1, \quad  j\in\Lambda, \quad  b\in K\cap_j^{-1}(\mathbb{B}(\mathbf{0},1)).\]

By Corollary~\ref{cor:sharp-Holder-norm}-(i), 
there exist some constants \(0<\delta_1<2\), \(0<A'\) such that,
for each \((b,j)\in K\times\Lambda\) of \(b\in K\cap\tau_j^{-1}(\mathbb{B}(\mathbf{0},1))\),
$$V_{K\cap \bar{B}_j(b,r)}(x)\leq \frac{A'}{\varkappa_0}\frac{\delta^{\nu}}{r^{q_0+2}}, \quad d_{\omega}(x,b)\leq\delta\leq \delta_1,0<r\leq 1.$$
The comparability \eqref{eq:compare-C}, inequality \eqref{eq:re-ext-compare-c} and Lemma~\ref{lem:equivalence-notions} in turn imply that, for each \((b,j)\in K\times\Lambda\) with the restrictions  $$b\in K\cap\tau_j^{-1}(\mathbb{B}(\mathbf{0},1)),\quad \|z-\tau_j(b)\|\leq\delta\leq\delta_1/l_2,\quad 0<r\leq 1,$$
with $$m_j(b):=\frac{\liminf_{y\to\partial B_j(b,2)}\rho_{1,b,j}(y)}{1+\|\rho_{1,b,j}\|_{B_j(b,2)}}\equiv\frac{4d}{1+4d}\quad\text{and}$$
$$ c_j(b,r):=\sup_{\mathbb{B}(\tau_j(b),2)}L_{\tau_j(K\cap \bar{B}_j(b,r))}\;,$$
we have the following chain of inequalities: 
$$\begin{aligned}
L_{\tau_j(K\cap \bar{B}_j(b,r))}(z)&\leq c_j(b,r) u_{\tau_j(K\cap \bar{B}_j(b,r));\mathbb{B}(\tau_j(b),2)}(z)\\
&\leq c_j(b,r)(u_{\tau_j(K\cap \bar{B}_j(b,r)),\rho_{1,b,j};\mathbb{B}(\tau_j(b),2)}(z)
-\inf_{B_j(b,2)}\rho_{1,b,j})\\
&=c_j(b,r) u_{\tau_j(K\cap \bar{B}_j(b,r)),\rho_{1,b,j};\mathbb{B}(\tau_j(b),2)}(z)\\
&\leq \frac{c_j(b,r)}{m_j(b)} (V_{K\cap \bar{B}_j(b,r)}(\tau_j^{-1}(z))+\rho_{1,b,j}(\tau_j^{-1}(z)))\\
&\leq \frac{c_j(b,r)}{m_j(b)}(\frac{A'(l_2\delta)^{\nu}}{\varkappa_0 r^{q_0+2}}+d \delta^2)\\
&\leq \frac{1+4d}{4d}C_1(\frac{A'}{\varkappa_0}\frac{l_2^{\nu}}{r^{2q_0+2}}+\frac{d}{r^{q_0}})\delta^{\nu}\\
&\leq\frac{1+4d}{4d}C_1(\frac{A'}{\varkappa_0}l_2^{\nu}+d)
\frac{\delta^{\nu}}{r^{2q_0+2}}.
\end{aligned}$$
In this chain of inequalities, \eqref{eq:cap-to-classicalSiciak-in-the-second} is used to uniformly bound \(c_j(b,r)\) from above.
Therefore, \(K\) has local \(\nu\)-H\"older continuity property of order \(2q_0+2\).
\end{proof}

\begin{remark}\label{rmk:local-mu-hcp} 
For a compact \(K\subset X\) and \(b\in K\), suppose for 
\(\mu_1\in(0,1]\), \(q_1>0\) and \textbf{some} holomorphic coordinate ball \((\Omega_1,\tau_1)\) centered at \(b\in K\), for some \(C'_1>0\),
$$\varpi_{\tau_1(K\cap \Omega_1)\cap \bar{\mathbb{B}}(\mathbf{0},r)} (\mathbf{0},\delta) \leq \frac{C'_1\delta^{\mu_1}}{r^{q_1}}, \quad 0<\delta\leq 1,\, 0< r \leq 1.$$
Let \((\Omega_2,\tau_2)\) be another holomorphic coordinate ball centered at \(b\). There exist \(r_1,r_2\in(0,1]\) such that \(\tau_1(\Omega_1)\supset \bar{\mathbb{B}}(\mathbf{0},r_1),\tau_2(\Omega_2)\supset \bar{\mathbb{B}}(\mathbf{0},r_2)\).
There exist \(0<r_3\leq r_1\) and \(l\in(1,\infty)\) such that
\(\tau_2^{-1}(\bar{\mathbb{B}}(\mathbf{0},r_2))\supset \tau_1^{-1}(\bar{\mathbb{B}}(\mathbf{0},r_3))\) and \(lr_3\leq r_2\) and
$${\mathbb{B}}(\mathbf{0},lr)\supset\tau_2\circ\tau_1^{-1}({\mathbb{B}}(\mathbf{0},r))\supset {\mathbb{B}}(\mathbf{0},r/l),\quad r\in(0,r_3].$$
$$\sup_{\bar{\mathbb{B}}(\mathbf{0},r_3)}L_{\tau_1(K\cap \tau_1^{-1}(\bar{\mathbb{B}}(\mathbf{0},r)))}^*\leq \frac{C_1'r_3^{\mu_1}}{r^{q_1}},\quad 0<r\leq \frac{r_3}{l^3}.$$
From this, by the inequality~\eqref{eq:cap-density-sufficient-loc},
$$\frac{Cap(\tau_1(K\cap \tau_1^{-1}(\bar{\mathbb{B}}(\mathbf{0},r))),\mathbb{B}(\mathbf{0},r_3))}{r^{nq_1}}\geq\frac{(2\pi)^n}{(C'_1 r_3^{\mu_1})^n},\quad
0<r\leq \frac{r_3}{l^3}.$$
By the monotonicity of the capacity, since 
\(\tau_2\circ\tau_1^{-1}(\bar{\mathbb{B}}(\mathbf{0},r_3))\supset\bar{\mathbb{B}}(\mathbf{0},r_3/l)\), 
we get
$$\frac{Cap(\tau_2(K\cap \tau_2^{-1}(\bar{\mathbb{B}}(\mathbf{0},lr))),\mathbb{B}(\mathbf{0},r_3/l))}{r^{nq_1}}\geq\frac{(2\pi)^n}{(C'_1 r_3^{\mu_1})^n},\quad
0<r\leq \frac{r_3}{l^3}.$$
By the inequality~\eqref{eq:capacity comparison in Cn}, 
since \(lr_3/l^3=r_3/l^2<2r_3/(l^2+l)<r_3/l\),
for some constant \(c'(r_3,l):=G(2r_3/(l^2+l))>0\),
for \(c:=(C'_ 1r_3^{\mu_1})^nc'(r_3,l)/(2\pi)^n>0\),
$$\sup_{\bar{\mathbb{B}}(\mathbf{0},r_3/l)}L_{\tau_2(K\cap \tau_2^{-1}(\bar{\mathbb{B}}(\mathbf{0},lr)))}^*\leq \frac{(C_1'r_3^{\mu_1})^n c'(r_3,l)}{(2\pi)^nr^{nq_1}}=\frac{c}{r^{nq_1}},\quad 0<r\leq \frac{r_3}{l^3}.$$

By the definition of relative extremal functions in Section~\ref{sec:appendix}, for \(0<r\leq r_3/l^3\), 
$$u_{\tau_2(K\cap \tau_2^{-1}(\bar{\mathbb{B}}(\mathbf{0},lr)));\mathbb{B}(\mathbf{0},r_3/l)}
\circ (\tau_2\circ\tau_1^{-1})
=u_{\tau_1(K\cap \tau_2^{-1}(\bar{\mathbb{B}}(\mathbf{0},lr)));\tau_1\circ\tau_2^{-1}(\mathbb{B}(\mathbf{0},r_3/l))}.$$
By monotonicity with respect to set inclusion, for \(0<r\leq r_3/l^3\), 
$$u_{\tau_1(K\cap \tau_2^{-1}(\bar{\mathbb{B}}(\mathbf{0},lr)));\tau_1\circ\tau_2^{-1}(\mathbb{B}(\mathbf{0},r_3/l))}
\leq u_{\tau_1(K\cap \tau_1^{-1}(\bar{\mathbb{B}}(\mathbf{0},r)));\mathbb{B}(\mathbf{0},r_3/l^2)}.$$
By the comparability inequality \eqref{eq:compare-C}, 
on \(\mathbb{B}(\mathbf{0},r_3/l^2)\) and for \(0<r\leq r_3/l^3\), 
$$u_{\tau_1(K\cap \tau_1^{-1}(\bar{\mathbb{B}}(\mathbf{0},r)));\mathbb{B}(\mathbf{0},r_3/l^2)}
\leq (\inf_{\partial\mathbb{B}(\mathbf{0},r_3/l^2)}L_{\tau_1(K\cap \tau_1^{-1}(\bar{\mathbb{B}}
(\mathbf{0},r)))})^{-1}L_{\tau_1(K\cap\tau_1^{-1}( \bar{\mathbb{B}}(\mathbf{0},r)))}.$$
We can estimate the coefficient on the right-hand side for \(0<r\leq r_3/l^3\) as 
$$(\inf_{\partial\mathbb{B}(\mathbf{0},r_3/l^2)}L_{\tau_1(K\cap \tau_1^{-1}(\bar{\mathbb{B}}
(\mathbf{0},r)))})^{-1}
\leq (\inf_{\partial\mathbb{B}(\mathbf{0},r_3/l^2)}L_{\bar{\mathbb{B}}(\mathbf{0},r)})^{-1}
=(\log{\frac{r_3}{l^2r}})^{-1}
\leq(\log{l})^{-1}.$$
Accordingly, on \(\mathbb{B}(\mathbf{0},r_3/l^2)\) and for \(0<r\leq r_3/l^3\), 
$$\begin{aligned}
&u_{\tau_2(K\cap \tau_2^{-1}(\bar{\mathbb{B}}(\mathbf{0},lr)));\mathbb{B}(\mathbf{0},r_3/l)}
\circ (\tau_2\circ\tau_1^{-1})
\leq(\log{l})^{-1}L_{\tau_1(K\cap\tau_1^{-1}( \bar{\mathbb{B}}(\mathbf{0},r)))}.
\end{aligned}$$


Therefore, by the comparability inequality \eqref{eq:compare-C}, 
for each \(\|z\|\leq\delta<r_3/l^2\) and \(0<r\leq r_3/l^3\), the value $L_{\tau_2(K\cap\tau_2^{-1}(\bar{\mathbb{B}}(\mathbf{0},lr)))}\circ(\tau_2\circ\tau_1^{-1})(z)$ is less than equal to
$$(\sup_{\mathbb{B}(\mathbf{0},r_3/l)}L_{\tau_2(K\cap \tau_2^{-1}(\bar{\mathbb{B}}(\mathbf{0},lr)))}^*)
u_{\tau_2(K\cap \tau_2^{-1}(\bar{\mathbb{B}}(\mathbf{0},lr)));\mathbb{B}(\mathbf{0},r_3/l)}\circ(\tau_2\circ\tau_1^{-1})(z),$$
which is less than equal to
$$\frac{c}{r^{nq_1}}\frac{1}{\log{l}}L_{\tau_1(K\cap\tau_1^{-1}( \bar{\mathbb{B}}(\mathbf{0},r)))}(z)
\leq\frac{c}{r^{nq_1}}\frac{1}{\log{l}}\frac{C'_1\delta^{\mu_1}}{r^{q_1}}.
$$
Using the variable \(t=lr\), we get 
$$L_{\tau_2(K\cap\tau_2^{-1}(\bar{\mathbb{B}}(\mathbf{0},t)))}((\tau_2\circ\tau_1^{-1})(z))
\leq \frac{cC'_1l^{(n+1)q_1}}{\log{l}}\frac{\delta^{\mu_1}}{r^{(n+1)q_1}},\;\|z\|\leq\delta<\frac{r_3}{l^2},\;
t\in(0,\frac{r_3}{l^2}].$$
Therefore, for constants \(0<\mu\leq1\) and \(q>0\), if a compact set \(K\subset X\) has the weak local \(\mu\)-HCP of order \(q\) at \(b\in K\), then \(K\) has local \(\mu\)-HCP of order \((n+1)q\) at \(b\in K\). By the same reason, if \(K\) has the weak local \(\mu\)-HCP of order \(q\), then \(K\) has local \(\mu\)-HCP of order \((n+1)q\).   
\end{remark}

\subsection{Proof of the third theorem}
Using Theorem~\ref{thm:characterization-hcp} and Theorem~\ref{thm:characterization-hcp-hcp}, we can prove that a compact subset in \(X\) with the weak H\"older continuity property is \((C^{\alpha},C^{\alpha'})\)-regular (defined in Definition~\ref{defn:DMN-regular}) for some constants \(\alpha\) and \(\alpha'\) in \((0,1]\) depending on the subset. 
\begin{proof}[Proof of Theorem~\ref{thm:UPC-intro}]
There exist constants \(\mu\in (0,1]\) and \(q>0\) such that \(K\) has weak local \(\mu\)-H\"older continuity property of order \(q\). Then by Theorem~\ref{thm:characterization-hcp-hcp}, \(K\) has a uniform density in capacity of order \(nq\) and for \(\mu':=2\mu/(q+2)\), \(V_K\) is \(\mu'\)-H\"older continuous. 

\cite{Mc34}, Mcshane extension, extends a H\"older continuous function from a compact subset to \(X\) with the same H\"older exponent and the same H\"older constant. Since \(X\) is compact, each connected component of \(X\) has a finite diameter with respect to the Riemannian distance \(d_{\omega}\). Therefore, 
for a given bounded subset \(\{\phi_\gamma\}_{\gamma\in\Gamma}\) in \(C^{\mu'}(K)\), each \(\phi_{\gamma}\) can be extended to a function \(\tilde{\phi}_{\gamma}\) in \(C^{\mu'}(X)\) such that \(\{\tilde{\phi}_\gamma\}_{\gamma\in\Gamma}\) is bounded in \(C^{\alpha}(X)\). Let \(\mu'':=(\mu')^2/(\mu'+2+nq)\). 
The direct computation gives $$\mu''=\frac{4\mu^2}{(q+2)(2\mu+(q+2)(nq+2))}.$$
Then by Remark~\ref{rmk:exponent-W}, for each \(\gamma\in\Gamma\), \(V_{K,\tilde{\phi}_{\gamma}}\) is \(\mu''\)-H\"older continuous. 

Remark~\ref{rmk:exponent-W} is a consequence of the proof of Theorem~\ref{thm:characterization-hcp-hcp}-(ii). By the proof Theorem~\ref{thm:characterization-hcp-hcp}-(ii), we know that \(\{V_{K,\tilde{\phi}_{\gamma}}\}_{\gamma\in\Gamma}\) is bounded in \(C^{\alpha}(X)\). The reason is as follows. Proposition~\ref{prop:W} is used in the proof Theorem~\ref{thm:characterization-hcp-hcp}-(ii) to extend the H\"older continuity of \(V_{K,\tilde{\phi}_{\gamma}}\) on \(K\) to the global H\"older continuity. Proposition~\ref{prop:W} uses Lemma~\ref{kisleg}. In Lemma~\ref{kisleg}, the constant \(c_0\) depends only on \((X,\omega)\) and the constants \(c_1\) and \(\delta_1\) depends only on \((X,\omega,\|V_{K,\tilde{\phi}_{\gamma}}\|_X)\). By Lemma~\ref{lem:weight-vs-unweight}, 
\(\|V_{K,\tilde{\phi}_{\gamma}}\|_X\leq \|V_K\|_X+\|\phi_{\gamma}\|_K\). Therefore, \(\{V_{K,\tilde{\phi}_{\gamma}}\}_{\gamma\in\Gamma}\) is bounded in \(C^{\alpha}(X)\). Since \(V_{K,\tilde{\phi}_{\gamma}}=V_{K,{\phi}_{\gamma}}\) for each \(\gamma\in\Gamma\), we know that \(\{V_{K,{\phi}_{\gamma}}\}_{\gamma\in\Gamma}\) is bounded in \(C^{\alpha}(X)\). Therefore, \(K\) is \((C^{\mu'},C^{\mu''})\)-regular. 
\end{proof}

\section{Appendix}
\label{sec:appendix}
In this section, we give the relation between extremal functions on \(X\) and locally defined weighted relative extremal functions in \(\mathbb{C}^n\). Nguyen in \cite{Ng24} proved this relation for a compact K\"ahler manifold, and we adjust his proof for a compact Hermitian manifold. 

Let \(a\in X\). There is a holomorphic coordinate ball $(\Omega, \tau)$ in \(X\) centered at \(a\), which is the relatively compact restricted chart in another holomorphic chart (if that bigger chart is \((U,f)\), \(U\) is open in \(X\), \(f\) is biholomorphic from \(U\) onto an open set in \(\mathbb{C}^n\), \(\Omega\Subset U,f|_\Omega=\tau\)): \(\Omega\) is open in \(X\), \(\tau\) is biholomorphic with
\[\label{eq:c-ball}\tau: \Omega \to \mathbb{B}:=\mathbb{B}(\mathbf{0},1) \subset \mathbb{C}^n, \quad \tau(\Omega)=\mathbb{B},\quad \tau(a)=\mathbf{0}.\] 
There exist non-negative bounded smooth strictly psh functions $\rho_1,\rho_2\in PSH(\Omega) \cap C^\infty(\Omega)\cap L^{\infty}(\Omega)$ 
such that \(\liminf_{x\to\partial\Omega}\rho_1(x)>0\), 
\(\rho_1(a)=\rho_2(a)=0\) and 
$dd^c\rho_1\leq\omega \leq dd^c \rho_2$.
(\(\rho_j(x):=A_j|\tau(x)|^2,\;x\in\Omega\), for some \(A_j\in(0,\infty)\) are such functions, as \(f|_{\Omega}=\tau\), \(\Omega\Subset U\).)  

For $E\Subset \mathbb{B}$, the (zero-one) relative extremal function of \(E\) on \(\mathbb{B}\) is 
\[\label{eq:zero-one-C}\notag
	u_{E}(z) =u_{E;\mathbb{B}}(z):= \sup \left\{ v(z) : v\in PSH(\mathbb{B}), v|_E \leq 0, v\leq  1 \right\},\quad
z\in\mathbb{B}.\]
\cite[Proposition~5.3.3]{Kl91} gives a relation of \(L_F\) and \(u_F\) for a compact non-pluripolar set \(F\) in \(\mathbb{B}\) with \(C_1:=\inf_{\partial\mathbb{B}} L_{F}\in (0,\infty)\) and \(C_2:=\sup_{\mathbb{B}} L_F^*=\sup_{\mathbb{B}}L_F\in (0,\infty)\) as
\[\label{eq:compare-C}
	C_1 \, u_F(z) \leq L_F(z) \leq C_2 u_F(z), \quad{z\in\mathbb{B}}.\]
For a bounded function \(\phi:\mathbb{B}\to\mathbb{R}\),
the weighted relative extremal function of \((E,\phi)\) on \(\mathbb{B}\) is
\[\label{eq:zero-one-CW}\notag
	u_{E,\phi} (z) =u_{E,\phi;\mathbb{B}}(z):=\sup \{ v(z) : v\in PSH(\mathbb{B}), \, v|_E \leq \phi|_E, v\leq  \phi+ 1 \},\quad z\in\mathbb{B}.\]
Like Lemma~\ref{lem:weight-vs-unweight}-(i),
by the definitions, (\(\|\cdot\|_{\mathbb{B}}\) denotes the supremum norm on \(\mathbb{B}\))
\[\label{eq:re-ext-compare-c}
	u_{E} +\inf_{\mathbb{B}}\phi \leq u_{E,\phi} \leq  (1+\|\phi\|_{\mathbb{B}})u_E + \sup_E \phi  \quad\text{when } \phi\geq0.\]
Like \eqref{eq:monotonicity}, for $E_1\subset E_2\Subset \mathbb{B}$ and $-\infty<-\|\phi_1\|_{\mathbb{B}}\leq \phi_1 \leq \phi_2\leq\|\phi_2\|_{\mathbb{B}}<\infty$, 
\[\label{eq:re-ext-monotonicity}
	u_{E_2,\phi_1} \leq u_{E_1,\phi_1} \leq u_{E_1, \phi_2}.\]

The following lemma tells a relation between an extremal function on \(X\) and the pullback of a weighted relative extremal functions on \(\mathbb{B}\) by a chart on \(X\). 
\begin{lem} \label{lem:equivalence-notions} Let $\emptyset\neq K$ be a compact non-pluripolar subset of \(X\) and \(a\in K\). Let \((\Omega,\tau)\) be the holomorphic chart centered at \(a\) given in \eqref{eq:c-ball} and \(\rho_1, \rho_2\) be the functions below \eqref{eq:c-ball}. If \(K\subset{\Omega}\), then there exist constants 
$0<M=M(K,\Omega,\|\rho_2\|_{\Omega})$ and 
$0<m=m(\liminf_{x\to\partial\Omega}\rho_1(x),\|\rho_1\|_{\Omega})$ 
such that,
for $\hat{\rho_j}:= \rho_j\circ \tau^{-1}$,
$$
 V_K\circ\tau^{-1}(z)+\hat{\rho_2}(z)
 \leq M\, u_{\tau(K),\hat{\rho_2}} (z), \quad z\in\mathbb{B},
$$
$$
 m\, u_{\tau(K),\hat{\rho_1}} (z) \leq 
 V_K\circ\tau^{-1}(z)+\hat{\rho_1}(z),\quad z\in\mathbb{B}.
$$
In fact, the second inequality holds even when \(K\) is pluripolar.
\end{lem}
\begin{proof} We use the proof of \eqref{eq:compare-C} in \cite[Proposition~5.3.3]{Kl91}. 
$\sup_\Omega V_K^*=\|V_K^*\|_{\Omega} <\infty$ by Proposition~\ref{prop:pluripolarextremal} as \(K\) is not pluripolar. 
Let $v\in PSH(X,\omega)$, $v|_K\leq 0$. 
Then $v +\rho_2$ is in $ PSH({\Omega})$. Take $M:=\|V_K^*\|_{\Omega} +\|\rho_2\|_{\Omega}+1$. 
Since $M\geq1$ and $\rho_2 \geq 0$, 
$$(v +\rho_2)|_{K} \leq {\rho_2}|_K\leq M\rho_2|_K, 
\quad v+\rho_2\leq M\leq M(1+\rho_2).$$ 
These mean \(v+\rho_2\leq M u_{\tau(K),\hat{\rho_2}}\circ \tau\). Since \(v\) was an arbitrary competitor for \(V_K\), 
we get the first inequality.

To obtain the second inequality,
Let $m':=\liminf_{x\to\partial\Omega} \rho_1(x)>0$. 
Take an open neighborhood \(O\) of \(K\). Since \(O\) is not pluripolar, \(V_O=V_O^*\in PSH(X,\omega)\) by Propositions~\ref{prop:pluripolarextremal}, \ref{prop:elementary}. 
Let $v \in PSH(\mathbb{B})$ be a competitor for \(u_{\tau(K),\hat{\rho_1}}\). In other words, $v|_{\tau(K)} \leq \hat{\rho_1}|_{\tau(K)}$, $v\leq \hat{\rho_1}+1$. Define 
$$v_O|_\Omega:=\max\{\frac{m'}{1+\|\rho_1\|_{\Omega}}v\circ\tau-\rho_1,V_O\}|_\Omega,\quad
v_O|_{X\setminus\Omega}:=V_O|_{X\setminus\Omega}.$$ 
By Proposition~\ref{prop:envelope}-(i), as
$$\limsup_{x\to \partial\Omega} 
[\frac{m'
}{1+\|\rho_1\|_{\Omega}} v\circ \tau (x)
- \rho_1(x)] \leq 
m'- \liminf_{x\to\partial \Omega} \rho_1(x)
=0\leq V_{O},$$ 
we have \(v_O\in PSH(X,\omega)\).
Let $m:=m'/(1+\|\rho_1\|_\Omega)$. 
Since \(m\in (0,1)\), $v\circ\tau|_K\leq\rho_1|_K$ and $\rho_1\geq0$, we have $(mv\circ\tau-\rho_1)|_K\leq0$. Also, \(V_O=0\) on \(K\). Therefore, \(v_O\leq V_K\), \(mv\circ\tau-\rho_1\leq V_K|_\Omega\). 
Since \(v\) was an arbitrary competitor for \(u_{\tau(K),\hat{\rho_1}}\), the second inequality holds.
\end{proof}


\begin{thebibliography}{000000000}



\bibitem[Ah26]{Ah26} H. Ahn, {\it Characterization of continuity of Siciak-Zaharjuta extremal functions on compact Hermitian manifolds}, preprint (2026), arXiv:2604.01885. 

\bibitem[AT84]{AT84}
H.~J. Alexander\ and\ B.~A. Taylor, {\it Comparison of two capacities in ${\bf C}\sp{n}$}, Math. Z. {\bf 186} (1984), no.~3, 407--417.







\bibitem[BT82]{BT82} E. Bedford and B. A. Taylor, {\it A new capacity for plurisubharmonic functions.} Acta Math. {\bf149} (1982), 1--40.





\bibitem[BD12]{BD12}
{R. Berman and J.-P. Demailly, \it Regularity of
plurisubharmonic upper envelopes in big cohomology classes. \rm
Proceedings of the Symposium ``Perspectives in Analysis, Geometry
and Topology'' in honor of Oleg Viro (Stockholm University, May 2008),
Eds. I.~Itenberg, B.~J\"oricke, M.~Passare, Progress in Math.\ {\bf 296}, Birkh\"auser/Springer, Boston (2012) 39--66.}








\bibitem[BK07]{BK07}
Z. B\l ocki\ and\ S. Ko\l odziej, {\it On regularization of plurisubharmonic functions on manifolds}, Proc. Amer. Math. Soc. {\bf 135} (2007), no.~7, 2089--2093.






\bibitem [De94]{De94}
{J.-P. Demailly, \it Regularization of closed positive
currents of type $(1,1)$ by the flow of a Chern
connection. \rm  Aspects of Math. Vol. E26, Vieweg (1994), 105--126.}

\bibitem [De16]{De16}
S. Dinew, {\it Pluripotential theory on compact Hermitian manifolds,} Annales de la Facult\'e des Sciences de Toulouse. {\bf XXV} (2016), no.~1, 91--139. 

\bibitem[DDGHKZ]{DDGHKZ}
J.-P. Demailly, S. Dinew, V. Guedj, P. Hiep, S. Ko\l odziej and  A. Zeriahi, {\it H\"older continuous solutions to Monge-Amp\`ere equations}, J. Eur. Math. Soc. (JEMS) {\bf 16} (2014), no.~4, 619--647.




\bibitem[DMN17]{DMN} T.-C. Dinh, X. Ma\ and\ V. Nguy\^{e}n, {\it Equidistribution speed for Fekete points associated with an ample line bundle}, Ann. Sci. \'{E}c. Norm. Sup\'{e}r. (4) {\bf 50} (2017), no.~3, 545--578. 

\bibitem[DK12]{DK12} 
S. Dinew, S. Ko\l odziej, {\it Pluripotential estimates on compact Hermitian Manifolds}, Advances in geometric analysis, 69--86, Adv. Lect. Math, {21}, International Press of Boston, Somerville, MA, 2012.

\bibitem[GL22]{GL22} V. Guedj, C.~H. Lu, {\it Quasi-plurisubharmonic envelopes 2: Bounds on Monge-Amp{\`e}re volumes}, Algebraic Geometry. {\textbf{9}} (2022), no.~6, 688--713.





\bibitem[GZ17]{GZ17}
V. Guedj\ and\ A. Zeriahi, {\it Degenerate complex Monge-Amp\`ere equations}, EMS Tracts in Mathematics, 26, Eur. Math. Soc., Z\"{u}rich, 2017. 

\bibitem[GKZ08]{GKZ08}
V. Guedj, S. Ko\l odziej, A. Zeriahi,
{\it H\"older continuous solutions to Monge-Amp\`ere equations}, 
Bull. Lond. Math. Soc. {\textbf{40}} (2008), 1070--1080.

\bibitem[Kl91]{Kl91}
M. Klimek, {\it Pluripotential theory}, London Mathematical Society Monographs. New Series, 6, The Clarendon Press, Oxford University Press, New York, 1991.

\bibitem[Ko05]{Ko05}
S. Ko{\l}odziej, {\it The complex Monge-Amp\`ere equation and pluripotential theory}, Mem. Amer. Math. Soc. {\bf 178} (2005), no.~840, x+64 pp.

\bibitem[Ko08]{Ko08}
S. Ko{\l}odziej, {\it H\"older continuity of solutions to the complex Monge-Amp\`ere equation with the right hand
side in \(L^p\). The case of compact K¨ahler manifolds}, Math. Ann. {\textbf{342}} (2008), 379--386. 


\bibitem[KN18]{KN18}
S. Ko\l odziej\ and\ N.~C. Nguyen, {\it H\"older continuous solutions of the Monge–Amp\`ere equation on compact Hermitian manifolds}, Annales de l'Institut Fourier. {\bf 644} (2018),
no.~7, 2951--2964.


\bibitem[KN19]{KN19}
S. Ko\l odziej\ and\ N.~C. Nguyen, {\it Stability and regularity of solutions of the Monge-Amp\`ere equation on Hermitian manifolds}, Adv. Math. {\bf 346} (2019), 264--304.

\bibitem[KN23]{KN23}
S. Ko\l odziej, N.-C. Nguyen, {\em The Dirichlet problem for the complex Monge-Ampere equation on Hermitian manifolds with boundary,} \rm Calc. Var. {\bf 62}, 1 (2023).


\bibitem[KN25]{KN25}
S. Ko\l odziej, N.-C. Nguyen, {\em Polar sets for \(m\)-subharmonic functions on compact Hermitian manifolds}, 
preprint (2025), arXiv:2511.01159.




\bibitem[LPT21]{LPT21}
C.~H. Lu, T.~T. Phung\ and\ T.~D. T\^{o}, {\it Stability and H\"{o}lder regularity of solutions to complex Monge-Amp\`ere equations on compact Hermitian manifolds}, Ann. Inst. Fourier (Grenoble) {\bf 71} (2021), no.~5, 2019--2045.

\bibitem[Mc34]{Mc34}
E.~J. McShane, {\it Extension of Range of functions}, Bull. Amer. Math. Soc. {\textbf{40}} (1934), no.~12, 837--842. 


\bibitem[Ng24]{Ng24}
N.~C. Nguyen, {\it Regularity of the Siciak-Zaharjuta extremal function on compact Kähler manifolds},
Trans. Am. Math. Soc. {\textbf{377}} (2024) no.~11, 8091--8123.

\bibitem[NST11]{NST11}
C. Nour, R.~J. Stern and J. Takche, {\it Validity of the Union of
Uniform Closed Balls Conjecture}, J. Convex Anal. {\textbf{18}} (2011) no.~2, 589--600.













\bibitem[PSS12]{PSS12} 
D.~H. Phong, J. Song, J. Sturm, {\it Complex Monge-Amp\`ere equations}, Surveys in Differential Geometry, {\textbf{17}}, no. 1, 327--410, International Press of Boston, 2012.










\bibitem[Si62]{Si62}
J. Siciak, {\it On some extremal functions and their applications in the theory of analytic functions of several complex variables,} Trans. Amer. Math. Soc. {\bf 105} (1962), 322--357.


\bibitem[Si81]{Siciak81}
J. Siciak, {\it Extremal plurisubharmonic functions in ${\bf C}\sp{n}$}, Ann. Polon. Math. {\bf 39} (1981), 175--211.








\bibitem[Vu19]{Vu19}
D.-V. Vu, {\it Locally pluripolar sets are pluripolar}, Int. J. Math. {\bf 30} (2019), no.~13, 1950029. 


\bibitem[Za76]{Za76}
V. Zaharjuta, {\it Extremal plurisubharmonic functions, orthogonal polynomials, and the Bern\v{s}te\u{\i}n-Walsh theorem for functions of several complex variables,} Ann. Polon. Math. {\bf 33} (1976/77), no.~1-2, 137--148.

\bibitem[Z20]{Z20}
A. Zeriahi, {\it Remarks on the modulus of continuity of subharmonic functions}, preprint (2020), arXiv:2007.08399.


\end{thebibliography}
\end{document}